\documentclass[11pt]{amsart}
\usepackage{hyperref,amsmath,a4wide,graphicx,wasysym,amssymb,mathtools}
\numberwithin{equation}{section}
\usepackage[all]{xy}
\usepackage{tikz}
\usetikzlibrary{arrows,calc,shapes,decorations.pathreplacing,cd}

\newtheorem{lemma}[equation]{Lemma}
\newtheorem{proposition}[equation]{Proposition}

\theoremstyle{definition}
\newtheorem{definition}[equation]{Definition}
\newtheorem{example}[equation]{Example}
\newtheorem{corollary}[equation]{Corollary}

\theoremstyle{remark}
\newtheorem*{remark}{Remark}
\newcommand{\mb}[1]{{\mathbf #1}}
\newcommand{\mc}[1]{{\mathcal #1}}
\newcommand{\tc}[1]{\textcolor{#1}}
\DeclarePairedDelimiter{\ceil}{\lceil}{\rceil}
\title{Pair component categories for directed spaces}
\author{Martin Raussen} 
\address{Department of
  Mathematical Sciences, Aalborg University, Skjernvej 4A,
  DK-9220 Aalborg {\O}st, Denmark} 
\email{raussen@math.aau.dk}
\urladdr{http://people.math.aau.dk/~raussen/} 
\thanks{The author thanks the Hausdorff Research Institute for Mathematics in Bonn, Germany, for its hospitality during two visits as part of the programme Applied and Computational Algebraic Topology in 2017 that allowed him to begin thinking about and discussing the topics dealt with in this paper.\\
A preliminary version of this article was presented at the Abel-Symposium 2018 in Geiranger, Norway. Support is gratefully acknowledged.\\
Thanks are also due to the referees who rightly suspected errors in a previous version and who suggested several improvements of the presentation.}

\begin{document}
\begin{abstract}
The notion of a homotopy flow on a directed space was introduced in \cite{Raussen:07} as a coherent tool for comparing spaces of directed paths between pairs of points in that space with each other. If all directed maps along such a 1-parameter deformation preserve the homotopy types of path spaces, such a flow and the parameter maps are called inessential.

For a directed space, one may consider various categories whose objects are pairs of reachable points to which a functor associates the space of directed paths between them. The monoid of all inessential maps acts on such a category by endofunctors leaving the associated path spaces invariant up to homotopy. We construct a pair component category as quotient category: it has as objects pair components along which the homotopy type is invariant -- for a coherent and transparent reason.

This paper follows up \cite{FGHR:04,GH:07,Raussen:07} and removes some of the restrictions for their applicability. At least in several examples, it gives reasonable results for spaces with non-trivial directed loops. If one uses homology equivalence instead of homotopy equivalence as the basic relation, it yields an alternative to computable versions of ``natural homology''  introduced in \cite{DGG:15} and elaborated in \cite{Dubut:17}. It refines, for good and for evil, the stable components introduced and investigated in \cite{Ziemianski:18}. 
\end{abstract}

\maketitle

\section{Introduction}
\subsection{Directed spaces and spaces of directed paths}
A \emph{directed space} (or \emph{d-space} for short) \cite{Grandis:01,Grandis:09} is a topological space $X$ together with a subset $\vec{P}(X)\subset P(X)=X^I$ of \emph{directed paths} (or \emph{d-paths}) satisfying reasonable properties: $\vec{P}(X)$ includes all constant paths, it is closed under concatenation $\alpha *\beta$ of d-paths $\alpha$ and $\beta$ and under non-decreasing reparametrizations. The set $\vec{P}(X)$ is given the compact-open topology inherited from $P(X)=X^I$. A map $f:X\to Y$ is a \emph{d-map} if $f(\vec{P}(X))\subset\vec{P}(Y)$. 

Particularly important d-spaces are the directed interval $\vec{I}$ with $\vec{P}(\vec{I})$ consisting of all non-decreasing maps $I\to I$, and the d-spaces $\vec{I}^n$ which are the building blocks for cubical complexes - the geometric realizations of pre-cubical sets with d-paths that are cubewise non-decreasing. Cubical complexes are the underlying d-spaces for Higher Dimensional Automata, models for true concurrency introduced and investigated by Pratt and van Glabbeek \cite{Pratt:91,Glabbeek:91,Glabbeek:06}; cf Section \ref{sss:pcs} for details. 

For two points $x,y\in X$, we consider the subspace $\vec{P}(X)_x^y\subset\vec{P}(X)$ of all d-paths with source $x$ and target $y$. The point $y$ is \emph{reachable} from $x$ if this subset is non-empty; and $(x,y)$ is then called a reachable pair.  A d-path $p$ and a reparametrization $q=p\circ\varphi$ with $\varphi:\vec{I}\to\vec{I}$ a surjective non-decreasing map are reparametrization equivalent. The symmetric and transitive closure of this relation is called reparametrization equivalence \cite{FR:07}. Equivalence classes, the so-called \emph{traces}, are the elements of \emph{trace spaces} $\vec{T}(X)_x^y$ with the quotient topology under the natural projection $p:\vec{P}(X)_x^y\to\vec{T}(X)_x^y$. In many cases, and in particular for cubical complexes $X$, these projection maps are homotopy equivalences \cite{Raussen:09}.
 
Unlike in classical topology, the topology of path and trace spaces may vary depending on the pair of end points, even for path-connected spaces, since the reverse path of a d-path is, in general, not a d-path; for simple examples cf Section \ref{s:ex}. In this article, we will partition, not the d-space $X$ itself (this was done in \cite{FGHR:04} and \cite{GH:07}), but the subspace $\vec{X^2}\subset X\times X$ of reachable pairs (cf Section \ref{sss:extcat}) into coherent ``components'' along which the homotopy type of the path and trace spaces have to be invariant.  

\subsection{A motivating example}\label{ss:Dubut}
Reading the impressive and comprehensive thesis \cite{Dubut:17} by J\'{e}r\'{e}my Dubut, an example (p.\ 162) with graphical representation in Figure 1 caught my attention:

\begin{center}\begin{figure}[h]\label{fig:Dub1}
\begin{tikzpicture}
  \draw (0,0) -- (4,0) -- (4,2) -- (2,2) -- (2,4) -- (0,4) -- (0,0);
\draw (2.5,2.5) -- (4.5,2.5) -- (4.5,4.5) -- (2.5,4.5) -- (2.5,2.5);
\draw (0,2) -- (2,2);
\draw (2,0) -- (2,2);
\node at (1,1) {$A$};
\node at (1,3) {$B_2$};
\node at (3,1) {$B_1$};
\node at (3.5,3.5) {C};
\draw[line width=0.6mm, color=red] (4,0) -- (4,2);
\draw[line width=0.6mm, color=blue] (0,4) -- (2,4);
\draw[line width=0.6mm, color=red] (2.5,2.5) -- (2.5,4.5);
\draw[line width=0.6mm, color=blue] (2.5,2.5) -- (4.5,2.5);
\end{tikzpicture}
\caption{The cubical complex $D$}
\end{figure}
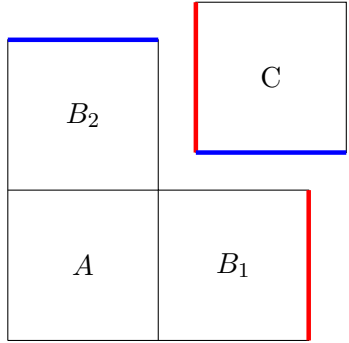\end{center}

The cubical complex $D$ (geometric realization of a pre-cubical set, cf Section \ref{sss:pcs}, Definition \ref{df:pcs}) consists of four 2-cells $A, B_1, B_2, C$. Remark that the cell $C$ shares a face with both cells $B_1$ and $B_2$. As a result, path spaces between points in cell $A$ and points in cell $C$ have different homotopy types depending on their \emph{relative} positions: The space $\vec{P}(D)_{[A,(x_1,x_2)]}^{[C,(y_1,y_2)]}, \; 0\le x_i, y_i\le 1$,
\begin{itemize}
\item is empty if $x_1>y_1$ and $x_2>y_2$
\item has two contractible components if $x_1\le y_1$ and $x_2\le y_2$
\item is contractible else;
\end{itemize}
cf Figure 2.

\begin{center}\begin{figure}[h]\label{fig:Dubu2}
\begin{tikzpicture}
  \draw (0,0) -- (4,0) -- (4,2) -- (2,2) -- (2,4) -- (0,4) -- (0,0);
\draw (2.5,2.5) -- (4.5,2.5) -- (4.5,4.5) -- (2.5,4.5) -- (2.5,2.5);
\draw (0,2) -- (2,2);
\draw (2,0) -- (2,2);
\node at (1,1) {$A$};
\node at (1,3) {$B_2$};
\node at (3,1) {$B_1$};
\node at (3.5,3.5) {C};
\draw[line width=0.6mm, color=red] (4,0) -- (4,2);
\draw[line width=0.6mm, color=blue] (0,4) -- (2,4);
\draw[line width=0.6mm, color=red] (2.5,2.5) -- (2.5,4.5);
\draw[line width=0.6mm, color=blue] (2.5,2.5) -- (4.5,2.5);
\node at (1,0.6) {$\blacksquare$};
\node at (2.7,2.7) {$\bullet$};
\end{tikzpicture}
\begin{tikzpicture}
  \draw (5,0) -- (9,0) -- (9,2) -- (7,2) -- (7,4) -- (5,4) -- (5,0);
\draw (7.5,2.5) -- (9.5,2.5) -- (9.5,4.5) -- (7.5,4.5) -- (7.5,2.5);
\draw (5,2) -- (7,2);
\draw (7,0) -- (7,2);
\node at (6,1) {$A$};
\node at (6,3) {$B_2$};
\node at (8,1) {$B_1$};
\node at (8.5,3.5) {C};
\draw[line width=0.6mm, color=red] (9,0) -- (9,2);
\draw[line width=0.6mm, color=blue] (5,4) -- (7,4);
\draw[line width=0.6mm, color=red] (7.5,2.5) -- (7.5,4.5);
\draw[line width=0.6mm, color=blue] (7.5,2.5) -- (9.5,2.5);
\draw[color=red] (6,0.6) -- (9,1.11);
\draw[color=red] (7.5,3.61) -- (8,3.7);
\node at (6,0.6) {$\blacksquare$};
\node at (8,3.7) {$\bullet$};
\end{tikzpicture}

\bigskip\begin{tikzpicture}
  \draw (0,0) -- (4,0) -- (4,2) -- (2,2) -- (2,4) -- (0,4) -- (0,0);
\draw (2.5,2.5) -- (4.5,2.5) -- (4.5,4.5) -- (2.5,4.5) -- (2.5,2.5);
\draw (0,2) -- (2,2);
\draw (2,0) -- (2,2);
\node at (1,1) {$A$};
\node at (1,3) {$B_2$};
\node at (3,1) {$B_1$};
\node at (3.5,3.5) {C};
\draw[line width=0.6mm, color=red] (4,0) -- (4,2);
\draw[line width=0.6mm, color=blue] (0,4) -- (2,4);
\draw[line width=0.6mm, color=red] (2.5,2.5) -- (2.5,4.5);
\draw[line width=0.6mm, color=blue] (2.5,2.5) -- (4.5,2.5);
\draw[color=blue](1,0.6) -- (1.46,4);
\draw[color=blue] (3.96,2.5) -- (4,2.8);
\node at (1,0.6) {$\blacksquare$};
\node at (4,2.8) {$\bullet$};
\end{tikzpicture}
 \begin{tikzpicture}
  \draw (5,0) -- (9,0) -- (9,2) -- (7,2) -- (7,4) -- (5,4) -- (5,0);
\draw (7.5,2.5) -- (9.5,2.5) -- (9.5,4.5) -- (7.5,4.5) -- (7.5,2.5);
\draw (5,2) -- (7,2);
\draw (7,0) -- (7,2);
\node at (6,1) {$A$};
\node at (6,3) {$B_2$};
\node at (8,1) {$B_1$};
\node at (8.5,3.5) {C};
\draw[line width=0.6mm, color=red] (9,0) -- (9,2);
\draw[line width=0.6mm, color=blue] (5,4) -- (7,4);
\draw[line width=0.6mm, color=red] (7.5,2.5) -- (7.5,4.5);
\draw[line width=0.6mm, color=blue] (7.5,2.5) -- (9.5,2.5);
\draw[color=red] (6,0.6) -- (9,1.23);
\draw[color=red] (7.5,3.73) -- (9,4);
\draw[color=blue](6,0.6) -- (6.34,4);
\draw[color=blue] (8.84,2.5) -- (9,4);
\node at (6,0.6) {$\blacksquare$};
\node at (9,4) {$\bullet$};
\end{tikzpicture}
\caption{Homotopy types of path spaces depend on end points}
\end{figure}
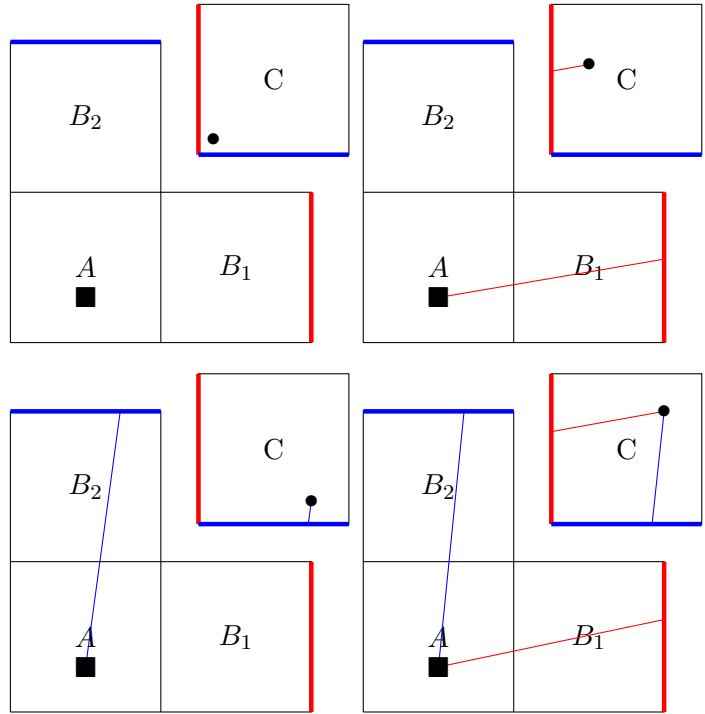\end{center}

Observe that a d-path with source in $[C,(y_1,y_2)]$ and target in $[C,(y_1',y_2')]$ hence does \emph{not} induce a homotopy equivalence between 
$\vec{P}(D)_{[A,(x_1,x_2)]}^{[C,(y_1,y_2)]}$  and $\vec{P}(D)_{[A;(x_1,x_2)]}^{[C;(y_1',y_2')]}$ by extension if $y_1<x_1\le y_1'$ and $x_2\le y_2$ (or if $y_2<x_2\le y_2'$ and $x_1\le y_1$); cf Figure 2. In particular, the only d-paths $\sigma$ within $D$ that, by extension, induce homotopy  equivalences on \emph{all} non-empty path spaces $\vec{P}(D)_{[A,(x_1,x_2)]}^{[C,(y_1,y_2)]}$ are the trivial ones: $y_1=y_1', y_2=y_2'$; cf Figure 3. Similarly for d-paths connecting various points within  $A$; they are never weak isomorphisms in the parlance of \cite{GH:07,FGHMR:16}. As a consequence, the fundamental category $\vec{\pi}_1(D)$ of the cubical complex $D$ (cf Section \ref{sss:piX}) allows only a trivial system of weak isomorphisms (consisting solely of the contant paths).

\begin{center}\begin{figure}[h]\label{fig:Dubut3}
\begin{tikzpicture}
\draw (0,0) -- (4,0) -- (4,2) -- (2,2) -- (2,4) -- (0,4) -- (0,0);
\draw (2.5,2.5) -- (4.5,2.5) -- (4.5,4.5) -- (2.5,4.5) -- (2.5,2.5);
\draw (0,2) -- (2,2);
\draw (2,0) -- (2,2);
\node at (1,1) {$A$};
\node at (1,3) {$B_2$};
\node at (3,1) {$B_1$};
\node at (3,2.8) {$C$};
\draw[line width=0.6mm, color=red] (4,0) -- (4,2);
\draw[line width=0.6mm, color=blue] (0,4) -- (2,4);
\node at (1,0.6) {$\blacksquare$};
\fill[fill=magenta] (3.5,4.5) -- (3.5,3.1) -- (4.5,3.1) -- (4.5,4.5) -- (3.5,4.5);
\fill[fill=red] (2.5,3.1) -- (3.5,3.1) -- (3.5,4.5) -- (2.5,4.5) -- (2.5,3.1);
\fill[fill=blue] (3.5,2.5) -- (4.5,2.5) -- (4.5,3.1) -- (3.5,3.1) -- (3.5,2.5);
 \draw[-> , color=green, line width=0.5mm] (3.5,4) -- (3.8,4);
  \draw[-> , color=green, line width=0.5mm] (4,3.1) -- (4,3.4);
\end{tikzpicture}
\caption{No non-trivial d-path within $A$ or $C$ gives rise to a weak isomorphism}
\end{figure}
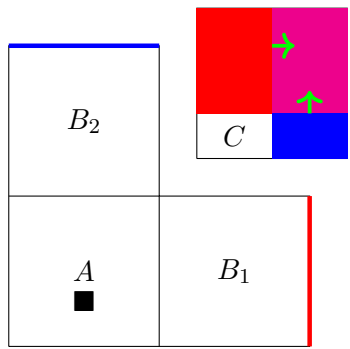\end{center}

\subsection{Previous work}
In order to obtain discrete invariants, previous work \cite{GH:07,FGHMR:16} studies localizations and component categories (cf Section \ref{ss:lc}) of the fundamental category $\vec{\pi}_1(X)$ (cf Section \ref{sss:piX}) of a d-space without non-trivial loops. So-called weak isomorphisms consisting of inessential d-paths (inducing equivalences on all non-empty path spaces) form systems of morphisms that are used for localizations or quotients. 

In Dubut's example in the previous section, it turns out that each component consists of a single point only; the components do not give rise to any state space reduction at all! K.\ Ziemia\'{n}ski \cite{Ziemianski:18} has recently suggested so-called \emph{stable} components to overcome this problem, cf also Section \ref{ss:relations}. 

A similar problem has been known for a long time for directed spaces with non-trivial directed loops. The simplest such space $X=\vec{S^1}$, the directed circle with counterclockwise directed paths, does not allow any non-trivial paths giving rise to weak isomorphisms either, cf \cite[Section 6.4.1]{FGHMR:16}.

\subsection{Contributions} The present paper takes a different approach resulting in reasonable finite component categories in both above mentioned cases (cf Section \ref{sss:compDubut} and Section \ref{sss:dc}); in particular, this is probably the first definition of a component category that can deal with directed spaces containing non-trivial directed loops.

The category of departure is no longer the fundamental category $\vec{\pi}_1(X)$ (with point as objects, cf Section \ref{sss:piX}) but its extension category $E\vec{\pi}_1(X)$ that has \emph{pairs} of reachable points within $\vec{X^2}:=\{ (x,y)\in X\times X|\; \vec{P}(X)_x^y\neq\emptyset\}$ as objects; the morphisms are pairs of d-homotopy classes of d-paths, cf. Section \ref{sss:extcat}. This category comes with a trace functor $\vec{T}(X)_*^*: E\vec{\pi}_1(X)\to \mathbf{Ho-Top}$ (cf Section \ref{sss:loc}) associating to a pair $(x,y)\in\vec{X^2}$ the trace space $\vec{T}(X)_x^y$. The aim is to identify, in a functorial way, pairs and d-homotopy classes in the extension category that give rise to the same data under $\vec{T}(X)_*^*$, up to (homotopy) equivalence.

The resulting component categories are results of a quotient formation that arises from an \emph{action} of a \emph{submonoid} of the monoid $\vec{C}(X,X)$ consisting of all d-maps from $X$ into itself. A d-map $f: X\to X$ is called \emph{inessential} if it is d-homotopic to the identity map $id_X$ via a  d-homotopy $H$ (called a \emph{homotopy flow}) that, moreover, is path space preserving (psp), ie it induces homotopy equivalences $\vec{T}(H_t): \vec{T}(X)_x^y\to\vec{T}(X)_{H_t(x)}^{H_t(y)},\; t\in I, x,y\in X$, on non-empty trace spaces, cf Section \ref{sss:homflow}, Definition \ref{def:dhom}. 

These inessential d-maps act on the extension category $E\vec{\pi}_1(X)$ by endo-func\-tors. In particular, the trace functor $\vec{T}(X)_*^*$ factors over the action of these inessential morphisms (up to isomorphisms). The arising components (objects in a component category) are the path components among the pairs of reachable points with respect to the effects of path space preserving homotopy flows on the end points: In which ways can a pair of source and target points be perturbed by a 1-parameter deformation of the entire d-space without changing the homotopy type (or another reasonable invariant) of the path space inbetween? 

To construct the pair component category, we make use of a localization and quotient process (cf Section \ref{ss:lc}) on a category with objects $\vec{X^2}$ and morphisms generated freely by those in $E\vec{\pi}_1(X)$ and, additionally, morphisms arising from the endo-functors induced by inessential d-maps on $X$ -- modulo several natural relations; for details cf Section \ref{sss:dext} and \ref{sss:isc}.

Several important properties of the arising pair component categories are collected in Section \ref{s:homflowprop}. In Section \ref{s:ex}, they are used in the investigation of a number of basic examples illustrating scope and results of the chosen approach. In particular, the pair component category of the directed circle $\vec{S^1}$ has two objects: the diagonal and its complement. Morphisms between them correspond either to the natural numbers $\mb{N}$ or to the augmented natural numbers $\mb{N}_{\ge 0}$, cf Section \ref{sss:dc}.

It is certainly out of scope to determine pair components of a general d-space and their category algorithmically. When the space in question is a cubical complex, ie the geometric realization of a pre-cubical set, (Section \ref{sss:pcs}), it is possible to find an approximation in the form of a so-called \emph{order component category}, cf Section \ref{s:precub}; usually much finer than the pair component category discussed previously: One considers only specific inessential d-maps (for details cf Section \ref{s:precub}) that preserve each cell of the complex. We verify that, for a cubical complex with finitely many cells, the localization and the quotient process, as above, using only these specific d-maps, gives rise to a \emph{finite} order component category. The pair component categories of cubical complexes are quotients of these order component categories.

Some of the constructions in this paper are borrowed from \cite{Raussen:07} and developed to suit new purposes. The reader should also compare K.\ Ziemianski's recent interesting paper \cite{Ziemianski:18} defining and investigating \emph{stable} components. Comments can be found in Section \ref{ss:relations}.
Pointers to future work, in particular investigating how far this approach can be made functorial, conclude this article in Section \ref{ss:dhe}.

\section{Categorical constructs. Towards components and their categories}
\subsection{A zoo of categories}

\subsubsection{Trace category}\label{sss:TX}
The \emph{trace category} \cite{Raussen:07} $\vec{T}(X)$ of a d-space $X$ has as
objects the elements of $X$. Morphisms from $x$ to $y$ are given by
$\vec{T}(X)(x,y):=\vec{T}(X)_x^y$. Identities are given by constant
traces, and composition by concatenation (up to
reparametrization; hence associative). The trace category is enriched in $\mathbf{Top}$.

A d-map $f:X\to Y$ between two d-spaces $X$ and $Y$ induces a functor
(of topologically enriched categories)
$\vec{T}(f):\vec{T}(X)\to\vec{T}(Y)$.

\subsubsection{Fundamental category}\label{sss:piX}
The fundamental category $\vec{\pi}_1(X)$ \cite{Grandis:01,FGR:06,Grandis:09,FGHMR:16} of a d-space $X$ is an ordinary category. It arises from the trace category  by identifying morphisms (ie traces) that are related by a directed homotopy (or a d-homotopy \cite{Grandis:01,Grandis:09}; this is not always the same notion!).  Morphisms are identified along
the path-component functor $\pi_0: \mathbf{Top}\to\mathbf{Set}$, giving rise to a quotient functor $\vec{\pi}_0: \vec{T}(X)\to\vec{\pi}_1(X)$.  
A d-map $f:X\to Y$ induces a functor $\vec{\pi}_1(f): \vec{\pi}_1(X)\to\vec{\pi}_1(Y)$.

\subsubsection{Extension and factorization categories}\label{sss:extcat}
Since we are interested in path spaces between given end points and their inter-relation, we need a category allowing for bookkeeping of \emph{both} start and end points.  The reachable pairs in a d-space $X$, ie those in $\vec{X^2}:=\{ (x,y)\in X\times X|\; \vec{P}(X)_x^y\neq\emptyset\}$, form the objects of the \emph{extension category}
$E\vec{T}(X)$ (called preorder category $\vec{D}(X)$ in \cite{Raussen:07}); cf.\ also
\cite{Mitchell:72,DGG:15,Dubut:17}) of the trace category $\vec{T}(X)$ (cf \ref{sss:TX}). It is considered as a \emph{full}
  subcategory of $\vec{T}(X)^{op}\times\vec{T}(X)$; an (extension) morphism has the
  form
  $(\alpha *, *\beta)\in E\vec{T}(X)$:
  \begin{equation}\label{eq:extmor}
  \xymatrix{(x,y)\ar[r]^{*\beta}\ar[d]_{\alpha *}\ar[rd]^{(\alpha *, *\beta)} & (x,y')\ar[d]^{\alpha *}\\
(x',y)\ar[r]^{*\beta} & (x',y')} 
\end{equation}
  
It was remarked by Fajstrup and Hess \cite{FH:18} that it is important to consider these categories \emph{together} with the subcategories which allow only right. resp.\ only left extensions in order to distinguish clearly different d-spaces (for example the one arising by reversing all arrows from the original one); for a careful analysis, consult \cite{CGM:18}.

  The extension category comes equipped with a functor
  $\vec{T}(X)_*^*:E\vec{T}(X)\to\mathbf{Top}$ with
  $\vec{T}(X)_*^*(x,y):=\vec{T}(X)_x^y$ and
  $\vec{T}(X)_*^*(\alpha *,*\beta)(\gamma)=\alpha * \gamma * \beta$.

More useful in the future is the extension category $E\vec{\pi}_1(X)$ of the fundamental category with morphisms
$(\alpha *, *\beta)\in E\vec{\pi}_1(X)((x,y),(x',y'))=\vec{\pi}_1(X)_{x'}^x\times\vec{\pi}_1(X)_y^{y'}$. It comes equipped with a functor $\vec{\pi}_1(X)_*^*:  E\vec{\pi}_1(X)\to\mathbf{Ho-Top}$ into the category of homotopy types (cf Section \ref{sss:loc}).

 A d-map $f:X\to Y$ between d-spaces induces a functor
  $E\vec{T}(f)$ between extension categories $E\vec{T}(X)\to E\vec{T}(Y)$ and a natural transformation
  $\vec{T}(f)_*^*$ from the functor $\vec{T}(X)_*^*$ to $\vec{T}(Y)_*^*\circ E\vec{T}(f)$ on $E\vec{T}(X)$. Likewise a functor $E\vec{\pi}_1(f): E\vec{\pi}_1(X)\to E\vec{\pi}_1(Y)$ and a natural transformation $\vec{\pi}_1(f)_*^*$ from $\vec{\pi}_1(X)_*^*$ to $\vec{\pi}_1(Y)_*^*\circ E\vec{\pi}_1(f)$ on $E\vec{\pi}_1(X)$.

Homotopy groups (of path spaces) require a base point. To allow the necessary categorical bookkeeping, one may consider the \emph{factorization categories} of the trace category, resp.\ the fundamental category, with traces, resp.\ d-homotopy classes of such as objects; cf \cite{Raussen:07,Dubut:17,CGM:18}. Although not essentially more difficult, we will not use factorization categories in the subsequent parts of this paper. 

\subsubsection{Endo-d-category}\label{sss:endo}
We will need further categories with the same objects (the set $\vec{X^2}$ of reachable pairs) but with
different morphisms: 

\begin{definition}\label{def:endo}
Let $X$ denote a topological space.
\begin{enumerate}
\item $C(X,X)$ denotes the topological monoid of all continuous self (or endo)-maps on $X$, equipped with the compact-open topology. 
\item If $X$ is a d-space, a \emph{d-map} $f\in C(X,X)$ is called
an \emph{endo-d-map}. Altogether, the endo-d-maps on $X$ form the topological submonoid $\vec{C}(X,X)\subset C(X,X)$ (under composition). 
\item Endo-d-maps give rise to the morphisms of the category
$d(X)$ with objects in $\vec{X^2}$, i.e., $d(X)((x,y)),(x',y'))$ is the space of all
d-maps  $f\in\vec{C}(X,X)$ satisfying $fx=x'$
and $fy=y'$.  Composition is given by composition of d-maps; the
identity map $id_X$ gives rise to all identity morphisms.
\end{enumerate} 
\end{definition}

The category $d(X)$ is topologically enriched. 

Remark that also
the endo-d-category comes with a functor
$d(X)_*^*:d(X)\to\mathbf{Top}$; on the objects, it is defined as for the extension category; on morphisms, it associates to $f: (x,y)\to (fx,fy)$ the map $\vec{T}(f) : \vec{T}(X)_x^y\to\vec{T}(X)_{fx}^{fy},\; [\sigma ]\to [f\circ\sigma ]$. 

In general, a d-map between d-spaces $X$ and $Y$ does \emph{not} induce a functor from $d(X)$ into $d(Y)$.

\subsubsection{d-extension category}\label{sss:dext}
Combining the morphisms from the
categories $d(X)$ and $E\vec{\pi}_1(X)$ yields the d-ex\-ten\-sion category 
of $X$:  The set of objects is again the set of reachable pairs in $\vec{X^2}$. The morphisms arise from a quotient of the category freely generated by the morphisms from $d(X)$ and from $E\vec{\pi}_1(X)$ by composition modulo the congruence relation making diagrams (\ref{eq:dE}) and (\ref{eq:flow}) below commute:

\begin{equation}\label{eq:dE}
\xymatrixcolsep{5pc}\xymatrix{(fx,fy)\ar[r]^{((f\circ\sigma)*, *(f\circ\tau ))} & (fx',fy')\\
(x,y)\ar[u]_f\ar[r]^{(\sigma *, * \tau )} & (x',y')\ar[u]_f},
\end{equation}
 for any $f\in\vec{C}(X,X), (x,y)\in\vec{X^2}, \sigma\in\vec{\pi}_1(X)_{x'}^x, \tau\in\vec{\pi}_1(X)_y^{y'}$; and

\begin{equation}\label{eq:flow}
\xymatrix{(fx, fy)\ar[r]^{*H(y)} & (fx, gy)\\
(x,y)\ar[u]^f\ar[r]^g & (gx,gy)\ar[u]_{H(x)*}}
\end{equation}
for any simple future d-homotopy (cf Definition \ref{def:dhom}) $H: X\times\vec{I}\to X$ from  $H_0=f$ to $H_1=g$, and for all $(x,y)\in \vec{X^2}$. Here $H(x)$ is the d-path arising by restricting $H$ to $x$.

\begin{remark}\label{rem:dext}
\begin{enumerate}
\item Imposing (\ref{eq:dE}) means that every endo-d-map $f\in\vec{C}(X,X)$ defines a functor from  $E\vec{\pi}_1(X)$ into itself. Altogether they define a monoid action of the monoid $\vec{C}(X,X)$  on $E\vec{\pi}_1(X)$ by such endo-functors.
\item Diagram (\ref{eq:dE}) encodes a \emph{coherent} $lr$-extension property (compare \cite{FGHR:04,GH:07}) of the morphisms in the subcategory of endo-d-maps with respect to the subcategory of morphisms in the subcategory of extensions.
\item As a consequence of (\ref{eq:dE}), every morphism in $dE\vec{\pi}_1(X)$ is a composition of just one endo-d-morphism and one extension morphism\\ $(\sigma *, *\tau)\circ f: (x,y)\to (fx,fy)\to (x',y')$.
\item If there is a simple d-homotopy $H$ from $H_0=f$ to $H_1=g$ such that $f(x)=H_0(x)=H_t(x)=H_1(x)=g(x)$ and $f(y)=H_0(y)=H_t(y)=H_1(y)=g(y),\; t\in I$, then, according to (\ref{eq:flow}), $f,g\in\vec{C}(X,X)$ give rise to the same morphism $f=g: (x,y)\to (fx,fy)=(gx,gy)$ in $dE\vec{\pi}_1(X)$.
\item The functors from the previous paragraphs can be aggregated to define a functor $dE\vec{\pi}_1(X)_*^*:dE\vec{\pi}_1(X)\to\mathbf{Ho-Top}$ (and this is why it makes sense to impose the commutativity of (\ref{eq:dE}) and (\ref{eq:flow}) above).
\item Note that, even if the d-space $X$ does not allow any non-trivial loops, the d-extension category may include non-identity endomorphisms arising from combinations of d-maps and of extensions. 
\item Inclusion of morphism sets defines functors from $E\vec{\pi}_1(X)$, resp.\ from $d(X)$ into $dE\vec{\pi}_1(X)$ -- all categories have the same objects given by $\vec{X^2}$. By (4) above, the latter functor is not faithful in general.
\end{enumerate} 
\end{remark}

\subsection{Homotopy flows and inessential d-maps}
\subsubsection{Homotopy flows}\label{sss:homflow}
We start by recalling elementary definitions about homotopy notions in directed algebraic topology: We distinguish the directed unit interval $\vec{I}$ with $\vec{P}(\vec{I})$ consisting of all non-decreasing self maps (d-paths) and the undirected unit interval with $\vec{P}(I)$ consisting of all constant maps (paths).

\begin{definition}\label{def:dhom}
Let $X$ and $Y$ denote two d-spaces.
\begin{enumerate}
\item A d-map $H: X\times\vec{I}\to Y$ is called a simple d-homotopy from $f=H_0$ to $g=H_1$ (also: a future d-homotopy from $f$ to $g$, or a past d-homotopy from $g$ to $f$) 
\item A d-map $H: X\times I\to Y$ is called a dihomotopy (or neutral d-homotopy).
\item Simple d-homotopies $H$ from $X$ to $X$ with $H_0=id_X$ (resp.\ $H_1=id_X$) are called future (resp.\ past) homotopy flows. Simple dihomotopies with $H_0=id_X$ (resp.\  $H_1=id_X$) are called neutral homotopy flows.
\end{enumerate}
\end{definition}
Both for a d-homotopy and for a dihomotopy, all level maps $H_t: X\to Y,\; t\in I$ are d-maps. Only for d-homotopies, every path  $H(x,-):I\to Y,\; x\in X$, is a d-path in $Y$.

The notion of homotopy flow \cite{Raussen:07} is meant to capture
some, but not all, of the properties of a flow for a dynamical system associated to a vector field.
Note that it is not demanded that the level maps $H_t: X\to X,\; 0\le t\le 1$, are invertible; they need
neither be injective nor surjective. The map $H$ is not supposed to
satisfy a group (or monoid) law either.

\paragraph{Concatenations and compositions.}
Two simple d-homotopies between endo-d-maps on a d-space $X$, say from $h$ to $f$, resp.\ from $f$ to $g$, can be concatenated to yield a d-homotopy from $h$ to $g$. In particular, 
a homotopy flow $H$ from $id_X$ to $f$ can be concatenated with a simple d-homotopy from $f$ to $g$ on $X$ to yield a homotopy flow from $id_X$ to $g$.

Homotopy flows on a given d-space $X$ can be composed in various ways
(cf \cite{Raussen:07}). Here we propose a generalized construction: For two simple future d-homotopies $G,H: X\times \vec{I}\to X$ let
$[G,H]: X\times \vec{I}^2\to X$ be given by $[G,H]((s,t);x)=G_s(H_t(x))$. In particular, $[G,H](0,0)=G_0\circ H_0, [G,H](s,0)=G_s\circ H_0, [G,H](0,t)=G_0\circ H_t$,\\ $[G,H](s,1)=G_s\circ H_1, [G,H](1,t)=G_1\circ H_t$ and $[G,H](1,1)=G_1\circ H_1$. Remark that this construction does not commute: in general, $[G,H]\neq [H,G]$. 

Any d-path $p\in\vec{P}(\vec{I}^2)_{(0,0)}^{(1,1)}$ provides a simple d-homotopy  
$[G,H]\circ p$ from $G_0\circ H_0$ to $G_1\circ H_1$. In particular, d-paths on the 1-skeleton of $\vec{I}^2$ joining $(0,0)$ and $(1,1)$ yield such simple d-homotopies  ``via'' $G_1\circ H_0$, resp.\ ``via'' $G_0\circ H_1$. In the special case of homotopy flows $G, H$, their composition $[G,H]$ provides homotopy flows ending at $G_1\circ H_1$ (via $G_1$, resp.\ via $H_1$). 

Similarly for past d-homotopies on $X$. For neutral homotopy flows the path $p$ does not have to be directed. The construction above can be generalized to yield a composition $[H_1,\dots , H_n]: X\times \vec{I}^n\to X$ of $n$ simple d-homotopies $H_i$.

\subsubsection{Psp homotopy flows and inessential d-maps}\label{sss:psp}
Ziemia\'{n}ski \cite[Definition 2.6]{Ziemianski:18} gives a list of very natural requirements to a family $\mc{F}$ of morphisms in the category $\mathbf{Top}$ to be considered as an equivalence system. We will concentrate here on the following particular cases (also considered in \cite{Ziemianski:18}): 
\begin{description}
\item[$\mc{F}=\mc{F}_{\infty}$] consists of all (weak) homotopy equivalences.
\item[$\mc{F}=\mc{F}_0$] consists of all maps inducing bijections on sets of path components ($\pi_0$).
\item[$\mc{F}=\mc{H}_{A,k}$] consists of all maps inducing isomorphisms on $H_n(-;A)$ for $n\le k$; $A$ denotes an abelian group.
\end{description}

In the following, all families $\mc{F}$ are supposed to be sandwiched between $\mc{F}_0$ and $\mc{F}_{\infty}: \mc{F}_0\subseteq\mc{F}\subseteq\mc{F}_{\infty}$. We are, first of all, interested in homotopy flows that preserve path spaces up to
 an $\mc{F}$-equivalence:

\begin{definition}\label{def:psp}
  \begin{enumerate}
  \item A d-map $f:X\to Y$ is called $\mc{F}$-\emph{path space preserving}
    ($\mc{F}$-psp for short) if
    $\vec{T}(f):\vec{T}(X)_{x_1}^{x_2}\to\vec{T}(Y)_{fx_1}^{fx_2}$ is an $\mc{F}$-equivalence for all pairs $(x_1,x_2)\in\vec{X^2}$.
    \item A d-homotopy (in particular, a homotopy flow) is called $\mc{F}$-psp if every d-map $H_t:X\to X,\; t\in I,$ is
      psp.
    \item An endo d-map $f:X\to X$ is called future/past/neutral
      $\mc{F}$-\emph{inessential} if there exists a future/past/neutral $\mc{F}$-psp homotopy
      flow with $H_0=id_X$ and $H_1=f$ (resp.\ $H_0=f$ and $H_1=id_X$).
  \end{enumerate}
  \end{definition}

In the following, we will write abbreviate the ``flavours''  \emph{future} with $\alpha =+$, \emph{past} with $\alpha =-$, and \emph{neutral} with $\alpha=0$. We may then talk about an $\alpha\mc{F}$ homotopy flow, resp.\ inessential map.

\begin{lemma}\label{lem:homflow} Let $X$ denote a d-space. 
\begin{enumerate}
\item The concatenation of an $\alpha\mc{F}$ homotopy flow on $X$ ending at $f$ with an $\alpha\mc{F}$ homotopy from $f$ to $g$ (cf Section \ref{sss:homflow}) yields an $\alpha\mc{F}$ homotopy flow ending at $g$.
\item The $\alpha\mc{F}$-inessential maps on a d-space $X$ are closed under composition; they form a submonoid $\Sigma_{\mc{F}}^{\alpha}(X)\subset \vec{C}(X,X)$ of the monoid of all endo-d-maps on $X$.
\item Consider the morphisms denoted $f$ in (\ref{eq:dE}) in the case where $f$ is $\mc{F}$-psp. Then the functor $dE\vec{\pi}_1(X)_*^*$ (cf Section \ref{sss:dext}, Remark \ref{rem:dext}) sends \emph{all} these morphisms into $\mc{F}$-equivalences. 
\item If $f$ is $+\mc{F}$-inessential via a $+\mc{F}$ homotopy flow $H$ keeping $x_1$ \emph{fixed} (i.e., $fx_1=x_1$), then all extensions $(c_{x_1},*H(x_2)): (x_1,x_2)\to (x_1,fx_2)$ starting at $x_1$ induce $\mc{F}$-equivalences. Likewise, if $f$ is $-\mc{F}$ inessential via a $-\mc{F}$ homotopy flow $H$ from $f$ to $id_X$ fixing $x_2$, then extensions $(H(x_1)*,c_{x_2}): (x_1,x_2)\to (fx_1,x_2)$ ending at $x_2$ induce $\mc{F}$-equivalences.
\end{enumerate}
\end{lemma}

\begin{proof}
The statements follow from the construction of compositions of (psp) homotopy flows in Section \ref{sss:homflow}, from Definition \ref{def:psp} and from (\ref{eq:flow}) in the case where one of the maps is the identity.

\end{proof}

  \begin{remark}\label{rem:homflow}
    \begin{enumerate}
\item In previous work (\cite{FGHR:04,GH:07}), attention was given to psp-properties of extension morphisms and, moreover, a pushout/pullback property encompassing that the psp property can be ``matched `` (on an individual basis) at start and end points. Asking for a psp homotopy flow means that there has to be a \emph{global} witness (the psp homotopy flow) for these psp properties. As the example in Section \ref{ss:Dubut} shows, it may be necessary to perturb start \emph{and} end point \emph{coherently together} to obtain ``constant'' path spaces (up to $\mc{F}$-equivalence). Hence, we are not going to compress the effects of inessential \emph{extension} morphisms but those of inessential \emph{d-maps}. 
\item The concepts ``psp homotopy flow'' and ``inessential d-map'' make also sense from a computer science applied perspective. A psp homotopy flow captures \emph{coherent} perturbations of all executions regardless of end points on a Higher Dimensional Automaton (HDA) in concurrency theory (cf \cite{Pratt:91,Glabbeek:06,FGHMR:16}).
    \end{enumerate}
  \end{remark}
  
\subsection{Localization and component categories}\label{ss:lc}
\subsubsection{Inessential subcategories: Definitions}\label{sss:isc} 
According to Lemma \ref{lem:homflow}, the $\mc{F}$-inessential d-maps on $X$ (cf Definition \ref{def:psp}.3) form submonoids $\Sigma^+_{\mc{F}}(X), \Sigma^-_{\mc{F}}(X)$, resp.\ $\Sigma^0_{\mc{F}}(X)$ (the neutral ones) of the monoid of all endo-d-maps on $X$. As such, they give rise to wide subcategories of $dE\vec{\pi}_1(X)$ that we call $\Sigma_{\mc{F}}^{\alpha}(X),\;\alpha = +,-,0$. Objects are always the sets $\vec{X^2}$, regardless the decorations $\alpha$ and $\mc{F}$.

A fourth flavour $\pm$ comes up as follows: Let\\ $\Sigma_{\mc{F}}^{\pm}(X)((x,y),(x',y')):=\{ f\in\Sigma_{\mc{F}}^+(X)((x,y),(x',y'))|\;\exists g\in\Sigma_{\mc{F}}^-(X)((x',y'),(x,y)):$\\ $\vec{T}(g\circ f): \vec{T}(X)_x^y\to\vec{T}(X)_x^y \mbox{ is } \mc{F-}\mbox{ homotopic to the identity map}\}$; ie   $\vec{T}(g\circ f)$
induces the identity map on all homotopy groups, on path components, resp.\ on a range of homology groups.

In the following (starting in Section \ref{sss:pcc}), we will concentrate on (mixed) subcategories $\Sigma_{\mc{F}}^{\alpha}E\vec{\pi}_1(X)\subset dE\vec{\pi}_1(X), \; \alpha = +, -, 0, \pm$, with
\begin{description}
\item[Objects] Pairs of points in $\vec{X^2}$.
\item[Morphisms] arise as finite compositions of inessential morphims in $\Sigma_{\mc{F}}^{\alpha}(X)$ with extension morphisms in $E\vec{\pi}_1(X)$ obeying to the relations (\ref{eq:dE}) and (\ref{eq:flow}); the latter for $f=id_X$. 
 \end{description}
As in Section \ref{sss:dext}, Remark \ref{rem:dext}, one should think of the monoid $\Sigma_{\mc{F}}^{\alpha}(X)$ as ``acting'' on the extension category $E\vec{\pi}_1(X)$ by endo-functors, this time leaving moreover the homotopy types of associated trace spaces invariant.

\subsubsection{Localization}\label{sss:loc}
Given a category $\mc{C}$ with subcategory $\Sigma$, the \emph{localization} \cite{Borceux:94} of $\mc{C}$ with respect to $\Sigma$ consists of a category $\mc{C}[\Sigma^{-1}]$ together with a functor $L: \mc{C}\to\mc{C}[\Sigma^{-1}]$ turning $\Sigma$-morphisms into isomorphisms, and such that a functor $F:\mc{C}\to\mc{D}$ factors uniquely through $L$ if and only if it sends all $\Sigma$-morphisms into isomorphisms. 

In practice, localization consists in adding formal inverses to the $\Sigma$-morphisms;
the morphisms in the localized category  are finite zig-zags consisting of 
morphisms in the category $\mc{C}$ and inverses of morphisms in the subcategory $\Sigma$. A prominent example defines the category $\mathbf{Ho-Top}$ as the result of localizing the subcategory whose morphisms $\Sigma=\mc{F}_{\infty}$ consists of all weak homotopy equivalences.

We will consider categories arising from a d-space $X$.  Of particular interest are the localized categories $\Sigma_{\mc{F}}^{\alpha}E\vec{\pi}_1(X)[\Sigma^{\alpha}_{\mc{F}}(X)^{-1}],$ $\alpha = +,-,0,\pm$: all morphisms arising from inessential d-maps (and no extension morphisms) are inverted. In these cases, we can draw several conclusions from (\ref{eq:dE}) and (\ref{eq:flow}) in Section \ref{sss:dext}:
\begin{lemma}\label{lem:loc}
\begin{enumerate}
\item The relations from \emph{(\ref{eq:dE})} and \emph{(\ref{eq:flow})} lead to reverse relations concerning morphisms $f^{-1}: (fx,fy)\to (x,y)$ for $f\in \Sigma_{\mc{F}}^{\alpha}(X)$ and $\mc{F}$ psp homotopy flows $H$ on $X$:
\begin{enumerate}
\item $(\sigma*,*\tau)\circ f^{-1}=f^{-1}\circ (f\sigma *,*f\tau)$.
\item $H(y)*\circ f^{-1}=*H(x): (fx,fy)\to (x,fy)$.
\end{enumerate}
\item Let $g$ denote an inessential d-map on $X$. Then $\vec{T}(g) $ indcuces bijections between morphism sets $E\vec{\pi}_1(X)((x,y),(x',y'))\to E\vec{\pi}_1(X)((gx,gy),(gx',gy'))$.
\end{enumerate}
\end{lemma}

\begin{proof}
\begin{enumerate}
\item follows directly from (\ref{eq:dE}) and (\ref{eq:flow}) in Section \ref{sss:dext}.
\item $E\vec{\pi}_1(X)((x,y),(x',y'))=\pi_0\vec{T}(X)_{x'}^x\times\pi_0\vec{T}(X)_y^{y'}\cong\pi_0\vec{T}(X)_{gx'}^{gx}\times\pi_0\vec{T}(X)_{gy}^{gy'}$\\ $=E\vec{\pi}_1(X)((gx,gy),(gx',gy'))$.
\end{enumerate}
\end{proof}

By the universal property characterizing localization, for $\mc{F}=\mc{F}_{\infty}$, the functors $\vec{T}(X)_*^*$ into $\mathbf{Top}$ and $\vec{\pi}_1(X)_*^*$ into $\mathbf{Ho-Top}$ (cf Section \ref{sss:extcat}) extend to give rise to functors from the localized categories into $\mathbf{Ho-Top}$ for which we will use the same notation. Similarly, for a wider class of equivalence systems, for the functors with target category $\mathbf{Ab}$ arising from composing with eg homology or $\mathbf{Set}$ from taking connected components.

\subsubsection{Component categories}\label{sss:cc} 
Going one step further, one may form from a category $\mc{C}$ with a subcategory $\Sigma$ a quotient category $\mc{C}/\Sigma$ together with a quotient functor $Q: \mc{C}\to \mc{C}/\Sigma$ sending morphisms in $\Sigma$ to \emph{identities} and such that a functor $F:\mc{C}\to\mc{D}$ factors uniquely through $Q$ if and only if it sends all $\Sigma $-morphisms to identities \cite{BBP:99,FGHMR:16}. The quotient category has as objects the \emph{path components} of $\mc{C}$-objects with respect to paths arising by composing morphism in $\Sigma$ and $\Sigma^{-1}$.  Morphisms in the quotient category are represented by morphisms in the category $\mc{C}[\Sigma^{-1}]$, ie by concatenations of morphisms in the original category and inverses of $\Sigma$-morphisms. Representatives can be composed if just their target, resp.\ source are situated in the same component (plug in an arbitrary $\Sigma$-morphism to obtain a morphism representing the composition); cf eg \cite{BBP:99,FGHMR:16} for details. In particular, every $\Sigma$-morphism in $\mc{C}$ represents an identity in the component category $\mc{C}/\Sigma$. 
 
By the universal properties, the quotient functor factors over the localization functor giving rise to a functor $\bar{Q}: \mc{C}[\Sigma^{-1}]\to \mc{C}/\Sigma$. This functor is not always  an \emph{equivalence} of categories; it is so if $\mc{C}$ is loop-free, ie if it contains only identities as endomorphisms;  cf \cite{GH:07,FGHMR:16}.

\subsubsection{Pair component categories}\label{sss:pcc}
Since objects are interpreted as path components, we will use the notation $\vec{\pi}_0(\mc{C};\Sigma)$ $=\mc{C}/\Sigma$ for relevant categories of pairs in our context. 

In the following, we describe a number of interesting pair component categories; in all of them, pairs $(x,y)$ and $(x',y')$ in $\vec{X^2}$ give rise to the same object (then called component) if and only if there exists a zig-zag $(x,y)\to (x_1,y_1)\leftarrow (x_2,y_2)\to\cdots\to (x_n,y_n)=(x',y')$ of $\Sigma$-morphisms. For $\Sigma =\Sigma^{\alpha}_{\mc{F}},\;\alpha =+,-,0,\pm$, those are induced by zig-zags of psp homotopies flows joining them. We will call the components in these categories \emph{future, past, neutral} resp.\ \emph{total} components.

By far the most important pair component categories arise as categories of components $\mc{C}/{\Sigma}$ with $\mc{C}=\Sigma_{\mc{F}}^{\alpha}E\vec{\pi}_1(X)$ and $\Sigma = \Sigma_{\mc{F}}^{\alpha}(X)$. As discussed above, objects (components) correspond to path components among reachable pairs along psp homotopy flows. Extension morphisms are identified if equivalent up to an inessential endo-d-map. The key relations go back to (\ref{eq:dE}) and (\ref{eq:flow}) -- for $f=id_X$ -- in Section \ref{sss:dext}:

For an inessential endo-d-map $f$ on $X$ and d-homotopy classes $\sigma , \tau$, the morphisms 
$(\sigma *, *\tau ):(x,y)\to (x',y')$ and $((f\circ\sigma)*, *(f\circ\tau)) :(fx,fy)\to (fx',fy')$ represent the same morphism in the component category.  Likewise, cf Lemma \ref{lem:loc}.2, morphisms $(\sigma' *, *\tau' ): (fx,fy)\to (fx',fy')$ arise from morphisms $(\sigma *, *\tau ): (x,y)\to (x',y')$ under $f$.

In the remaining part of the paper, we will use the shorter notation  $\vec{\pi}_0(X;\alpha , \mc{F})_*^*$ $=\vec{\pi}_0(\Sigma_{\mc{F}}^{\alpha}E\vec{\pi}_1(X);\Sigma_{\mc{F}}^{\alpha}(X))$, $\alpha =+,-,0,\pm$, for the \emph{future, past, neutral resp.\ total pair component categories}. Inclusion induces functors
\[\xymatrix{\vec{\pi}_0(X;\pm,\mc{F})_*^*\ar[r]\ar[d] & \vec{\pi}_0(X;+,\mc{F})_*^*\ar[d]\\
\vec{\pi}_0(X;-,\mc{F})_*^*\ar[r] & \vec{\pi}_0(X;0,\mc{F})_*^*
}.\]
These are onto on objects and, since every morphism arises from an extension, full on morphisms.

For $\mc{C}=\mathbf{Top}$ and $\Sigma=\mc{F}_{\infty}$, the quotient category $\mc{C}/\Sigma$ is the naive category of homotopy types. By the universal property characterizing quotient categories, the functors $\vec{T}(X)_*^*$ factors to give rise to functors from $\vec{\pi}_0(X;\alpha , \mc{F})_*^*$ to this category of homotopy types. Likewise for $\mc{F}_0, \mc{C}=\mathbf{Sets}$ and $\Sigma$ consisting of bijections; or $\mc{H}_{A;k}, \mc{C}=\mathbf{A-mod}$ and $\Sigma$ consisting of isomorphisms.

\begin{remark}
There are several variations on this theme that we are not going to follow up:
\begin{enumerate}
\item The full category $\mc{C}=dE\vec{\pi}_1(X)$ with the same subcategories $\Sigma = \Sigma_{\mc{F}}^{\alpha}(X)$; resulting in a pair component with the same objects as above but with an additional ``action'' of \emph{essential} endo-d-maps (the quotient of morphisms in $dE(X)$with respect to $\Sigma_{\mc{F}}^{\alpha}(X)$ and its inverses; see 3) below).
\item In both cases: Pair component categories with departure the extension category $E\vec{T}(X)$ of the trace category instead of that of the fundamental category.
\item Forgetting extensions, and considering $\mc{C}=d(X), \Sigma =\Sigma_{\mc{F}}^{\alpha}(X)$, morphisms given by an endo d-map $f$ and its compositions  $g\circ f$ and $f\circ g$ with an inessential such map $g$ are identified in the quotient category. 
\end{enumerate}
\end{remark}

\section{Properties of homotopy flows and components}\label{s:homflowprop}
\subsection{Components are path-connected}\label{ss:pathcon}
Any homotopy flow on a d-space $X$ yields, when restricted to a point or a pair of points, a path connecting the ends. This elementary observation implies:
\begin{lemma}\label{lem:pathcon}
Any component $C\subset\vec{X^2}$ is path-connected. For $\alpha = +, -, \pm$, two elements of the same component can be connected by a zig-zag of (pairs of) d-paths.
\end{lemma}

\subsection{Future components are future connected}
The following definition adapts similar ones from \cite{FGHR:04,Ziemianski:18} to the pair setting:
\begin{definition}
A subset $C\subset\vec{X^2}\subset X\times X$ is called \emph{future connected} if, for any two pairs $(x_1,y_1), (x_2,y_2)\in C$, there exist $(x,y)\in C$ and d-paths $\alpha_i, \beta_i\in\vec{P}(X),\; i=1,2,$  such that $(\alpha_i(t),\beta_i(t))\in C$ for all $t$, $\alpha_i(0)=x_i, \beta_i(0)=y_i, \alpha_i(1)=x, \beta_i(1)=y$.
\end{definition}
Past connectivity is defined similarly. We omit the decoration $\mc{F}$.

\begin{proposition}
Future components are future connected. Past components are past connected. Total components are both.
\end{proposition}

\begin{proof}
We give a proof of future connectivity; the other cases follow by considering the reverse directed space. There are essentially two zig/zag situations to consider (arrows indicate $+\mc{F}$-psp homotopy flows $G, H: X\times\vec{I}\to X$ by their restrictions to pairs of points): 
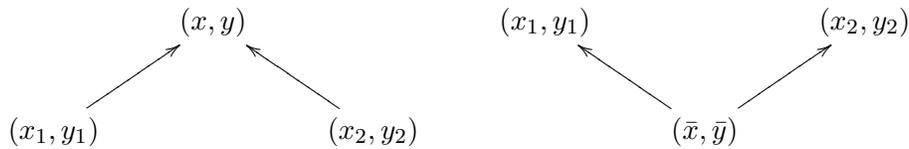
\[\xymatrix{& (x,y) & & (x_1,y_1) & & (x_2,y_2)\\(x_1,y_1)\ar[ur] & & (x_2,y_2)\ar[ul] & & (\bar{x},\bar{y})\ar[ul]\ar[ur] & }\]
In the case on the left, the d-paths joining $(x_i,y_i)$ with $(x,y)$ within $C$ can be chosen as restrictions of $G$ and $H$ to $(x_i,y_i)$. 

In the case on the right, there is a psp homotopy flow $G:X\times\vec{I}\to X$ with $G_1(\bar{x})=x_1, G_1(\bar{y})=y_1$ and a psp d-homotopy flow $H:X\times\vec{I}\to X$ with $H_1(\bar{x})=x_2, H_2(\bar{y})=y_2$. In the notation from Section \ref{sss:homflow}, let $x:=[G,H]((1,1);\bar{x})=G_1(H_1(\bar{x}))$ and $y:=[G,H]((1,1);\bar{y})=G_1(H_1(\bar{y}))$.  The dipaths $(\alpha_1(t), \beta_1(t))=[G,H]((1,t);(\bar{x},\bar{y}))$ connect $(x_1,y_1)$ with $(x,y)$ within $C$.
The dipaths $(\alpha_2(t), \beta_2(t))=[G,H]((t,1);(\bar{x},\bar{y}))$ connect $(x_2,y_2)$ with  $(x,y)$ within $C$.

The general zig-zag situation follows by an inductive argument.
 \end{proof}

\subsection{Regions fixed by (psp) homotopy flows}
\subsubsection{Pre-cubical sets. Cubical complexes}\label{sss:pcs}
\begin{definition}\label{df:pcs}
\begin{enumerate}
\item A \emph{pre-cubical set} $X$ (also called a $\Box$-Set) is a sequence of disjoint sets $X_n,\; n>0$, equipped with \emph{face maps} $d_i^{\alpha}: X_n\to X_{n-1},\; \alpha\in\{+,-\}, 1\le i\le n$, satisfying the pre-cubical relations: $d_i^{\alpha}d_j^{\beta}=d_{j-1}^{\beta}d_i^{\alpha}$ for $i<j$.\\
Elements of $K_n$ are called $n$-cubes, those of $K_0$ are called vertices.
\item The \emph{geometric realization} of a pre-cubical set $X$ is the d-space
\[|X|=\bigcup_{n\ge 0}X_n\times \vec{I}^n/_{[d_i^{\alpha}(c),x]\sim [c,\delta_i^{\alpha}(x)]}\] 
with $\delta_i^{\alpha}(x_1,\dots ,x_{n-1})=(x_1,\dots, x_{i-1}, s_{\alpha}, x_i,\dots , x_{n-1})$ and $s_{\alpha}= 0$ (resp.\ $1$) for $\alpha =-$ (resp.\ $\alpha =+$).
\item A path $p\in\vec{P}(|X|)$ is directed if there are $0=t_0<t_1<\dots <t_k=1$, cubes $c_i\in X_{n_i}$ and directed paths $p_i:[t_{i-1},t_i]\to \vec{I}^{n_i}$ with $p(t)=[c_i,p_i(t)]$ for $t\in [t_{i-1},t_i]$.
\end{enumerate}
\end{definition}

Pre-cubical sets are the underlying structure of a Higher-Dimensional Automaton \cite{Pratt:91,Glabbeek:91,Glabbeek:06,FGHMR:16}; those have moreover a coherent labelling of the 1-cubes.
In this section, we consider homotopy flows on the geometric realization $|X|$ of a pre-cubical set $X$, also called a \emph{cubical complex}. We will often just write $X$ for $|X|$.

\subsubsection{Homotopy flows and components}
Let $Y\subset X$ denote a subset of a d-space $X$. Its \emph{past} $\downarrow\! Y:= \{x\in X|\;\exists y\in Y: \vec{P}(X)_x^y\neq\emptyset\}$ consists of all elements in $X$ that can reach an element in $Y$. Its \emph{future} $\uparrow\! Y:=\{x\in X|\;\exists y\in|c|: \vec{P}(X)_y^x\neq\emptyset\}$ consists of all elements in $X$ that can be reached from an element in $Y$. Both contain $Y$.

A cube $c$ in a finite dimensional cubical complex $X$ is called a \emph{future branch cube}, resp.\  a \emph{past branch cube} \cite{Raussen:12} if there exist \emph{more than one} maximal cube containing it as a bottom boundary cube (iterated $d^-_*$), resp.\ top boundary cube (iterated $d^+_*$).

\begin{proposition}\label{prp:hfc}
\begin{enumerate}
\item For every (neutral) homotopy flow on $X$ and every future branch cube $c$ there exists $T>0$ such that $H(|c|\times [0,T])\subseteq\downarrow\! |c|$.
\item For every (neutral) homotopy flow on $X$ and every past branch cube $c$ there exists $T>0$ such that $H(|c|\times [0,T])\subseteq\uparrow\! |c|$.
\item A future (resp.\ past) homotopy flow on $|X|$ preserves future (resp.\ past) branch cubes.
\end{enumerate}
\end{proposition}

\begin{proof}
\begin{enumerate}
Let $C=\bigcup |d|$, the union of all cubes $d$ that have $c$ as a lower face. For every homotopy flow on $X$, since $|c|$ is compact, there exists $T>0$ such that $H(|c|\times [0,T])\subseteq\downarrow\! C$. Assume there exists $x=[c; (t_1,\dots t_k)]\in |c|$ and $0<t\le T$ such that $H_t(x)\in |d|\setminus |c|$ for $0<t\le T$ for a cube $d$ containing $c$ as a lower face. Let $e$ denote a maximal cube containing $c$ but not $d$ as a lower face.
Let $y=[e,(t_1,\dots ,t_k,\varepsilon_1,\dots \varepsilon_{n-k})]$ be contained in $|e|\setminus |d|, \varepsilon_i>0$; we assume without restriction that $c$ occupies the first $k$ coordinates. Observe that there exists a d-path $\sigma$ from $x$ to $y$. The d-path $H_t(\sigma )$ starts in $H_t(x)\in|d|\setminus |c|$ and ends in $H_t(y)$ which is contained in $|e|\setminus |d|$ for small $t$. Contradiction!\\
\item The same reasoning applied to the reverse d-structur on $X$ (d-paths replaced by reverse d-paths) yields the result for past branch cubes.
\item is an immediate consequence of 1.\ and 2.
\end{enumerate}
\end{proof}

The following result is straightforward; it turns out to be very useful in Section \ref{s:ex}:
\begin{lemma}\label{lem:invariant}
Let $Y\subseteq X$ denote a subspace of a d-space $X$.
\begin{enumerate}
\item A d-map $f: X\to X$ (resp.\ a dihomotopy $H: X\times I\to X$) that keeps $Y$ invariant ($f(Y)\subseteq Y$, resp.\ $H(Y\times I)\subseteq Y$), keeps also invariant
its closure $\bar{Y}$, its past $\downarrow\! Y$ and its future $\uparrow\! Y$.
\item A future (resp.\ past) homotopy flow $H: X\times\vec{I}\to X$ keeping $Y$ invariant keeps also invariant the complement $X\setminus\downarrow\! Y$ of its past (resp.\ the complement ($X\setminus\uparrow\! Y$) of its future.
\item Intersections of invariant sets are invariant.
\end{enumerate}
\end{lemma}

\begin{corollary}\label{cor:downarrow}
\begin{enumerate}
\item Any future (resp.\ past) homotopy flow $H: X\times\vec{I}\to X$ fixes a future (resp.\ past) branch point (=$0$-cell) $b\in X_0$.
\item Any future homotopy flow preserves the past $\downarrow b=\{x\in X| \vec{P}(X)_x^b\neq\emptyset\}$ of a future branch point $b$ and its complement  
$X\setminus\downarrow b$. 
\item Any past homotopy flow preserves the future $\uparrow b=\{x\in X| \vec{P}(X)_b^x\neq\emptyset\}$ of a past branch point $b$ and its complement  
$X\setminus\uparrow b$.
\item For any disjoint collections $B_1, B_2$ of future branch points, any future homotopy flow preserves subsets of the form $\bigcap_{b_i\in B_1}\downarrow\! b_i\cap\bigcap_{b_j\in B_2}(X\setminus\downarrow\! b_j).$
\item Similarly for past homotopy flows and collections of past branch points.
\end{enumerate}
\end{corollary}

For $x\in X$, divide the set $B(X)\subseteq X$ of \emph{all} future branch points into the set $B_1(x)$ containing all those in $\uparrow\! x$ and into $B_2(x)$ containing the remaining ones.  Let 
\begin{equation}\label{eq:downarrow}
E^+(x):= \bigcap_{b_i\in B_1(x)}\downarrow\! b_i\cap\bigcap_{b_j\in B_2(x)}(X\setminus\downarrow\! b_j).
\end{equation}
Obviously, $x\in E^+(x)$. Intersections over empty sets $B_i(x)$ are interpreted as the entire space $X$; in particular, $E^+(x)=X$ if $B_1(x)=B_2(x)=\emptyset$. Observe that the relation defined by  $x\simeq y\Leftrightarrow E^+(x)=E^+(y)$ is an equivalence relation on $X$. If $B_1(x)\neq\emptyset$, then $E^+(x)$ is obviously path-connected.

As an immediate consequence of Corollary \ref{cor:downarrow}(4-5), we obtain:

\begin{corollary}\label{eq:smallcomp}
Let $C\subset\vec{X^2}$ denote a future component, ie an object in $\vec{\pi}_0(X;+,\mc{F})_*^*$. Let $p_i: \vec{X^2}\to X, i=1,2,$ denote one of the two projections, and let $x\in X$.\\ If $x\in p_i(C)$, then $p_i(C)\subset E^+(x)$.
\end{corollary}

Similar statements apply to past components whose projections are contained in similarly defined sets $E^-(x)$ and to total components whose projections are contained in sets $E^+(x)\cap E^-(x)$.

Remark that we did not need a psp-property (preservation of path spaces) for the result in Corollary \ref{eq:smallcomp}.

\subsubsection{Psp homotopy flows and components}
The following property is an immediate consequence of the definitions (for a general d-space $X$):
\begin{lemma}\label{lem:noniso}
If two pairs $(x_1,y_1)$ and $(x_2,y_2)$ belong to the same pair component in $\vec{\pi}_0(X;\alpha , \mc{F})^*_*$, then
\begin{enumerate}
\item $\vec{T}(X)_{x_1}^{y_1}$ and $\vec{T}(X)_{x_2}^{y_2}$ are $\mc{F}$-equivalent (ie (weakly) homotopy equivalent, homology equivalent $\dots$ ).
\item Furthermore, there exist paths $p$ and $q$ in $X$ connecting $x_1$ with $x_2$ resp.\ $y_1$ with $y_2$ such that $\vec{P}(X)_{p(t)}^{q(t)}$ are all $\mc{F}$-equivalent.
For $\alpha = +, -, \pm$, these paths can be chosen as zig-zags of d-paths.
\end{enumerate}
\end{lemma}

The following easy observation is useful for many examples and replaces lr-conditions on left/right extensions \cite{FGHR:04,GH:07}, also called Ore conditions in \cite{Dubut:17}:

\begin{lemma}\label{lem:lr}
Let $H$ denote a future $\mc{F}$-psp homotopy flow, $x_1\in X, x_2:=H_1(x_1)$. For every $y_1\in X$ with $\vec{T}(X)_{x_1}^{y_1}\neq\emptyset$ there exists $y_2\in X$ with $\vec{T}(X)_{x_2}^{y_2}\neq\emptyset\neq\vec{T}(X)_{y_1}^{y_2}$ such that $(x_1,y_1)$ and $(x_2,y_2)$ are contained in the same $+\mc{F}$-component.\\
Similarly for past $\mc{F}$-psp homotopy flows.
\end{lemma}

\begin{proof}
The $\mc{F}$-psp property of $H$ makes sure that $\vec{T}(H_t):\vec{T}(X)_{x_1}^{y_1}\to\vec{T}(X)_{H_tx_1}^{H_tx2}, \; t\in I,$ and, in particular, $\vec{T}(H_1):\vec{T}(X)_{x_1}^{y_1}\to\vec{T}(X)_{x_2}^{y_2}$ is an $\mc{F}$-equivalence.
\end{proof}

\subsection{Products}\label{ss:prod}
Let $X_1, X_2$ denote d-spaces; we will consider our constructs for their product, the d-space $X=X_1\times X_2$. There are natural homeomorphisms $\vec{T}(X)\cong\vec{T}(X_1)\times\vec{T}(X_2)$ fibered over $\vec{X^2}\cong\vec{X_1^2}\times\vec{X_2^2}$, giving rise to natural homeomorphisms $\vec{T}(X)_{(x_1,x_2)}^{(y_1,y_2)}\cong\vec{T}(X_1)_{x_1}^{y_1}\times\vec{T}(X_2)_{x_2}^{y_2}$ on all fibers, $(x_1,x_2), (y_1,y_2)\in X$. Hence the extension category $E\vec{T}(X)$ is naturally isomorphic to $E\vec{T}(X_1)\times E\vec{T}(X_2)$; likewise $E\vec{\pi}_1(X)$ to $E\vec{\pi}_1(X_1)\times E\vec{\pi}_1(X_2)$.

A pair of endo-d-maps, $f_1$ on $X_1$ and $f_2$ on $X_2$, defines the endo-d-map $f_1\times f_2$ on $X$. If the maps $f_i$ are $\mc{F}$-psp, then so is $f_1\times f_2$. But not every (psp) d-map on $X$ is a product. In general, the inclusion of product maps induces monomorphisms $d(X_1)\times d(X_2)\hookrightarrow d(X)$ and likewise $\Sigma_{\mc{F}}^{\alpha}(X_1)\times \Sigma_{\mc{F}}^{\alpha}(X_2)\hookrightarrow \Sigma_{\mc{F}}^{\alpha}(X)$. After localization and quotient formation, one obtains from the isomorphism of categories above:

\begin{proposition}\label{prop:prod}
The quotient functor \[\vec{\pi}_0(X_1; \alpha,\mc{F})_*^*\times\vec{\pi}_0(X_2; \alpha,\mc{F})_*^*\to\vec{\pi}_0(X; \alpha, \mc{F})_*^*\] is \emph{onto} on objects and full on morphisms.
\end{proposition}

\section{Examples}\label{s:ex}
In this section, we walk through a number of simple, but key examples for which it is possible to determine pair component categories by elementary considerations.
\subsection{Intervals and hyperrectangles}

\subsubsection{An interval}\label{sss:int}
Let $J\subseteq\mb{R}$ denote an interval; it may be (half-)open or closed, bounded or unbounded. All trace spaces $\vec{T}(J)_x^y$, $(x,y)\in\vec{J}^2$, are contractible, and hence any endo-d-map on $J$ is automatically psp. Moreover, any two endo-d-maps $f, g: J\to J$ are psp-d-homotopic to $\max (f, g)$ and to $\min (f, g)$ -- by convex combination. In particular, they are all inessential with respect to every choice of $\alpha$ and $\mc{F}$. In fact, the space of endo-d-maps on $J$ is contractible.  

For every two pairs 
$(x_1,y_1)$, $(x_2,y_2)\in\vec{J}^2$ with $x_1\neq y_1$ there exists an endo-d-map $f: J\to J$ with $f(x_1)=x_2, f(y_1)=y_2$, for example a piecewise linear map with $f(t)=x_2, t\le x_1$, and $f(t)=y_2, t\ge y_1$. Hence, any two pairs $(x_1,x_2), (x_2,y_2)$ are contained in the same (unique) component, and the pair component category $\vec{\pi}_0(J; \alpha, \mc{F})$ is the trivial category with one object and one (identity) morphism -- for every choice of $\alpha$ and $\mc{F}$.

\subsubsection{Hyperrectangles and generalizations}
A hyperrectangle is a finite product $H=J_1\times\cdots\times J_n$ of intervals. It follows from Proposition \ref{prop:prod} that also the pair component category $\vec{\pi}_0(H; \alpha, \mc{F})$ is the trivial category with one object and one (identity) morphism.

Using the same reasoning as in Section \ref{sss:int}, one can show that pair component categories are trivial, more generally, for subspaces $K\subseteq\mb{R}^n$ satisfying the following property: For every pair $x,y\in K$, the lines connecting $x$ and $y$ with $\max (x,y)$, resp.\ with $\min (x,y)$ are contained in $K$.
 
\subsection{Directed graphs}  
\subsubsection{A branching} A \emph{future branching} graph is the pre-cubical set consisting of two 1-cubes $a$ and $b$ and three 0-cubes $d^-_1(a)=d^-_1(b), d^+_1(a)$ and $d^+_1(b)$.  Its geometric realizaton is homeomorphic to the subspace $B=\{( (x,y)\in [0,1]^2|\; xy=0\}\subset\mb{R}^2$ with induced directed topology. All non-trivial path spaces are contractible. Any future homotopy flow $H:B\times\vec{I}\to B$ has to fix the only future branch point, ie the origin $O=(0,0)$ (equal to its own past $\downarrow\! O$) as well as the path components of its complement (cf Corollary \ref{cor:downarrow}.4 and Lemma \ref{lem:pathcon}) $B_x=\{ (x,0)|\; x\in ]0,1]\}$ and $B_y=\{ (0,y)|\; y\in ]0,1]\}$; note that $H$ is automatically psp. It is elementary to see that the future pair category $\vec{\pi}_0(B; +, \mc{F})_*^*$ (suppressing identities, and regardless of $\mc{F}$) of the form 
{\small 
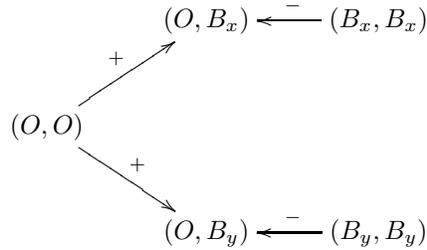
\[\xymatrix{ & (O,B_x) & (B_x,B_x)\ar[l]_{-} \\
(O,O)\ar[ru]^+\ar[rd]^+ & & \\
& (O,B_y) & (B_y,B_y)\ar[l]_{-} }\]
Each arrow represents a well-defined morphism (+ represented by an extension in the future, - by an extension in the past). This category is isomorphic to the extension category of the (poset) category with three objects $O, B_x, B_y$ and non-reversible relations given by $O<B_x$ and $O<B_y$. 

For the past pair component category, note that every reachable pair of points can be connected with the pair $(O,O)$ by a past psp homotopy flow. Hence the past pair category $\vec{\pi}_0(B; -, \mc{F})_*^*$ is the trivial category with one object and one identity morphism. 

The total pair category $\vec{\pi}_0(B; \pm , \mc{F})_*^*$ coincides with the future past category, the neutral pair category $\vec{\pi}_0(B;0, \mc{F})_*^*$ is also trivial.

\subsubsection{Particular directed graphs} We will consider a directed graph $G$ (a one-dimensional pre-cubical set), with the property that there is at most one directed path between two vertices.  Morphisms are thus either empty or singletons, and this gives rise to a partial order $\le$ on the vertices. Restrict this partial order to the (future) branch points (out-degree $>1$). For every branch point $b_i$, consider the subgraph $G_i:=\downarrow\! b_i\setminus \bigcup_{b_j<b_i}\downarrow\! b_j$; moreover the path-components $G_k^{top}$ of $G^{top}:=G\setminus \bigcup_j\downarrow\! b_j$; the latter correspond to top branches not including any branch point.

\begin{lemma}
The future components (objects of the future pair component category  $\vec{\pi}_0(G;+,\mc{F})_*^*$) are products of subgraphs of the form
\begin{itemize}
\item $G_i\times G_l$ for $b_i\leq b_l$
\item $G_i\times G_k^{top}$ if $G_k^{top}$ is reachable from $b_i$, or
\item $G_k^{top}\times G_k^{top}$.
\end{itemize}
Morphisms are inherited from the partial order on the branch points.
\end{lemma} 

\begin{proof}
It follows from Corollary \ref{eq:smallcomp} that a component must be contained in one of the products of subgraphs described above. On the other hand, any pair of points $(x_i,x_l)\in G_i\times G_l$ can be connected by a future homotopy flow (which is automatically psp)  with $(b_i,b_l)$ fixing the complement of $G_i\cup G_l$. Likewise any pair of points $(x_i,x_k)\in G_i\times G^{top}_k$ can be connected by a zig-zag of future homotopy flows with any other pair in $G_i\times G^{top}_k$.
\end{proof}

A similar result identifies past pair components (objects in $\vec{\pi}_0(G;-,\mc{F})_*^*$) with products of subgraphs between past branch points or of bottom branches. For total pair components (objects in  $\vec{\pi}_0(G;\pm ,\mc{F})^*_*$), both future and past branch points, top branches and bottom branches have to be taken into account. 

For the neutral pair components (objects in $\vec{\pi}_0(G;0,\mc{F})_*^*$), Proposition \ref{prp:hfc} tells us: Factors of components arise as subgraphs arising as zig-zags of d-paths between branch-points and connecting a past branch point with a future branch point (no intermediate future branch point) or connecting a future branch point with a past branch point (no intermediate past branch point) by a reverse d-path. Top branches and bottom branches are identified with a branch point and do not give rise to components.   

\begin{example}
 The directed graph representing the letter $M$ has one future branch vertex and two past ones. $M^{top}$ has two path components, and $\overrightarrow{\pi}_0(M;+,\mc{F})_*^*$ has five pair components. $M^{bot}$ has three path components; taking reachable pairs among the two branch vertices and these components leads to nine components in $\overrightarrow{\pi}_0(M;-,\mc{F})_*^*$. For the total pair component category, $\overrightarrow{\pi}_0(M;\pm,\mc{F})_*^*$, the three branch points and the four components of the complement stay invariant; this results in 15 pair components. The neutral pair component category $\overrightarrow{\pi}_0(M;0,\mc{F})_*^*$ is trivial.
 \end{example}
 
It is more complicated to determine components for graphs with several directed paths between certain pairs of vertices (thus including non-directed loops). For a simple example, consult Section
  \ref{sss:bdBox2} below.
  
\subsection{Simple cubical complexes without directed loops}
\subsubsection{Boundary of a square}\label{sss:bdBox2}
We consider the boundary $\partial\Box^2$ of the 2-cube $\Box^2$. Its geometric realization $X=\partial [0,1]^2$ decomposes into $A=\{ (0,0)\} , C=\{ (1,1)\}, B_1=\{ x,0)|\; x>0\}\cup\{ (1,y)|\; y<1\}$ and $B_2=\{ (0,y)|\; y>0\}\cup \{(x,1)|\; x<1\}$; we wish to show that its component category $\vec{\pi}_0(X;\alpha ,\mc{F})_*^*$ is given by 

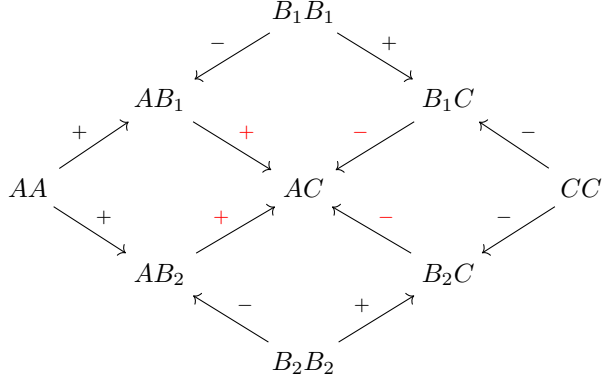
\begin{center}
\begin{tikzcd}
       & & B_1B_1 \arrow[ld, "-", swap] \arrow[rd, "+"] & & \\
& AB_1 \arrow[rd, "\tc{red}{+}"] & & B_1C \arrow[ld, "\tc{red}{-}", swap] & \\
AA \arrow[rd, "+"] \arrow[ru, "+"] & & AC & & CC \arrow[ld, "-", swap]\arrow[lu,"-" ,swap]\\
& AB_2 \arrow[ru, "\tc{red}{+}"] & & B_2C \arrow[lu, "\tc{red}{-}", swap] & \\
& & B_2B_2 \arrow[lu, "-", swap] \arrow[ru, "+"] & & 
    \end{tikzcd}
    \end{center}

 The path space $\vec{P}(\partial\Box^2)_A^C$ consists of two contractible components; all other non-empty path spaces are contractible. In particular, every psp-homotopy flow preserves both $A$ and $C$ and their complement $B_1\cup B_2$. The sets $B_1$ and $B_2$ are not connected to each other. By Lemma \ref{lem:pathcon}.2, no component can contain sources, resp.\ targets from more than one $B_i$ -- since projections of connected spaces are connected. 

As a consequence, there are three one element components $(A,C), (A,A)$ and $(C,C)$. All other components are contained in $(A,B_i), (B_i,C), (B_i,B_i),\; i=1,2$.  It is easy to find a (one-zig one-zag) psp homotopy flow connecting pairs within the given regions; for a formal justification, cf 
Section \ref{s:precub}. 

\noindent Morphisms that can be represented by future (resp.\ past) extensions are marked with $+$ (resp.\ $-$). A mark is in red if the extension (with target $AC$) does \emph{not} cover a homotopy equivalence of path spaces. The left square marked with $+$ and the right square marked with $-$ do \emph{not} commute; the mixed top and bottom squares do. Not surprisingly, this category is (isomorphic to) the extension category of the graph category on a directed graph with vertices
$A, B_1, B_2, C$ and directed arrows from $A$ to $B_i$ and from $B_i$ to $C$.

For $X=\partial\Box^2$, the past pair component category and also the total pair component category are isomorphic to the future pair component category : $X$ is d-homeomorphic to the space with the reverse d-path structure. Finally, also the neutral pair category yields the same result: A neutral psp-homotopy flow has to preserve the pair $(A,C)$ and hence $A$ and $C$; the rest of the argument is as above. 

The d-space $Y=I^2\setminus J^2$ with $J=]j_0,j_1[\subset I=[0,1]$ directed intervals (ie a square from which a minor square has been deleted, arising as model of a very simple HDA modeling mutual exclusion) yields the same results for the pair component categories. For the neutral category ($\alpha =0$), one may use the following argument by comparison: Consider a piecewise linear d-map $f: I\to I$ with $f(0)=f(j_0)=j_0, f(t)=t, j_0\le t\le j_1, f(j_1)=f(1)=j_1$. Then $f^2$ is a psp map from $Y$ into $X':=\partial [j_0,j_1]^2$ leaving $X'$ pointwise fix,  and $X'$ is d-homeomorphic to $X$. Convex combination of $id_Y$ and $f^2$ defines a neutral psp d-homotopy between these two maps, fixing $X'$ pointwise, and thus exhibits $X'$ as a directed deformation retract of $Y$ - with isomorphic neutral pair component category.

The deformation retraction above is not given by a future nor by a past d-homotopy. But it is not difficult to check that the pair component categories $\vec{\pi_0}(Y;\alpha , \mc{F})_*^*$ are  all mutually isomorphic.

\subsubsection{Swiss flag}\ The ``Swiss flag'' is an HDA that arises as a model for two processes only interacting via capacity one semaphores \cite{FGR:06,FGHMR:16}. It can be described by a Euclidean cubical complex (cf Section \ref{ss:Euclid}), a pre-cubical set whose geometric realization can be embedded into a ``cubical plane'', including a deadlock $D$ -- only constant paths with that source, a ``doomed region'' $d$ - every d-path starting in $d$ cannot leave $\downarrow\! D$ - and an unreachable region $u$.  It has a directed deformation retract \cite{Dubut:17} given by the directed graph
\begin{equation}\label{fig:swiss}\xymatrix{ & &  \ar @{}[dr] |{u} Y_1\ar[r]^{y_1} & C\\ & & U\ar[r]\ar[u] & Y_2\ar[u]_{y_2}\\ \ar @{}[dr] |{d}X_1\ar[uurr]^{b_1}\ar[r] & D &  & \\  A\ar[u]^{x_1}\ar[r]_{x_2} & X_2\ar[u]\ar[uurr]_{b_2} & &  &}\end{equation}
with two additional 2-cubes $d$ and $u$ glued in -- with vertices at $A,X_1,X_2,D$, resp.\ $U,Y_1,Y_2,C$. The geometric realization of this 2-complex will be called $S$. 

$X_1$ and $X_2$ represent the only future branch points in $S$. The intersection $\downarrow\! X_1\cap\downarrow\! X_2$ of their pasts is given by the vertex $A$. The intersection $\downarrow\! X_i\cap (S\setminus\downarrow\! X_j)$ of the past of one of them with the complement of the other's are the half-open 1-cells $x_i$ (including $X_i$, excluding $A$). The intersection $(S\setminus\downarrow\! X_1)\cap(S\setminus\downarrow\! X_2)$ of their complements has two connected components: one of them is the complement of the entire 2-cell $d$; the other consists of $d$ apart from the entire two lower 1-cells $x_i$ (corresponding to the doomed region). All of these subspaces are preserved by any future d-homotopy flow by Corollary \ref{cor:downarrow}. 

The pair $(A,C)$ is the only one representing a non-trivial path space.  Hence, as in \ref{sss:bdBox2}, none of the other areas can be moved into $C$ by a psp future homotopy flow. As a consequence, no point in $u\setminus y_1$ can be connected to a point on $y_1$ by a psp-future homotopy flow, since such a d-homotopy would move points on $y_2$ into $C$; and similarly when exchanging $y_1$ and $y_2$. Somewhat surprisingly, a psp future homotopy flow can ``discover'' the unreachable region, ie $u$ apart from its upper 1-cells.. The resulting future pair component category $\vec{\pi}_0(S; +, \mc{F})_*^*$ of the ``Swiss flag'' $S$ does not give rise to further compression; it is the extension category of  
\begin{equation}\label{fig:swiss2}\xymatrix{ & & b_1\cup y_1\ar[r] & C\\ & & u\ar[r]\ar[u] & b_2\cup y_2\ar[u]\\ x_1\ar[uurr]\ar[r] & d &  & \\  A\ar[u]\ar[r] & x_2\ar[u]\ar[uurr] & &  &}\end{equation}
where $x_i, y_i, b_i$ are half-open 1-cells, $d$ and $u$ contain only upper, resp.\ lower boundaries, and the lower und upper squares commute.

A similar procedure applies for the past pair component category $\vec{\pi}_0(S; -, \mc{F})_*^*$: just align $b_i$ with $x_i$ instead of $y_i$. For the total pair component category $\vec{\pi}_0(S; \pm , \mc{F})_*^*$, the (interiors of) the 1-cells $b_i$ give rise to separate components. For the neutral pair component category $\vec{\pi}_0(S; 0, \mc{F})_*^*$, the areas $x_i\cup b_i\cup y_i$ form one component.

\subsubsection{Matchbox}
The matchbox was discussed in \cite{Fahrenberg:04} and has been a test case for various constructs concerning homological constructions. It can be represented as the boundary $\partial\Box^3$ of a 3-box from which the interior of a lower face $\Box^2$ has been removed. An equivalent description is given as the product of $\partial\Box^2\times\vec{I}$ with an additional 2-cell $D$ glued in at the top.  We will use the latter representation and we will reuse notation from Section \ref{sss:bdBox2} concerning $\partial\Box^2$ in the following sense: Components from that case are replaced by their products with the half-open interval $\vec{[0,1[}$ -- and path spaces have the same homotopy types as for $\partial\Box^2$. In particular, only the pair $AC$ represents homotopy type $S^0$; all other pairs in the component category correspond to contractible path spaces. Nevertheless, compared to $\partial\Box^2$, the (entire) cell $D$ gives rise to additional components : There is no psp homotopy flow connecting a point from the complement of $D$ with $D$ itself: such a homotopy flow would push a path from $A$ to $C$ (corresponding to a non-contratible path space) to a path with target in $D$ -- with a contractible path space. 

As a consequence, additional components $AD, B_1D, B_2D, CD$ and $DD$ need to be added to the component categories of $\partial\Box^2$ from Section \ref{sss:bdBox2}, with obvious extension morphisms relating them to the other components and to each other. 

\subsubsection{Cubical ``spheres'' in higher dimensions}  A cube in the cubical decomposition of the boundary of an $n$-box $\partial\Box^n$ (the ``cubical $(n-1)$-sphere'') corresponds to an $n$-tupel $\mb{e}:=(e_1,\dots ,e_n)$ with $e_i\in\{ 0, 1,*\}$, and $\mb{e}\neq (*,\dots ,*)$. These $n$-tuples are partially ordered by the product order arising from $0\le *\le 1$ apart from $(*,\dots *, 0_i,*,\dots *)\not\le (*,\dots *, 1_i,*,\dots *)$.
Let $\mc{P}_n$ denote the category representing that partial order (with $3^{n}-1$ objects). Remark that $\partial\Box^n$ can be obtained as a directed deformation retract of the Euclidean complex $[-1,1]^n\setminus ]0,1]^n$. Hence, as will be shown in Section \ref{ss:Euclid}, products of (reachable) open subcubes will be contained in the same component. 

The homotopy types of trace spaces in $\partial\Box^n$
were determined by Ziemia\'{n}ski \cite{Ziemianski:18}: For $n$-tuples $\mb{e}\le\mb{f}$, let $d_i=f_i-e_i$ (where we replace $*$ with $0.5$). Let $n(\mb{e},\mb{f})$ denote the cardinality of the set $\{ i| d_i=1\}$.
\begin{lemma}\label{lem:hpytype}\emph{\cite{Ziemianski:18}}
Let $\mb{e}\le\mb{f}$. Then trace space $\vec{T}(\partial\Box^n)_{\mb{e}}^{\mb{f}}$ is contractible if there exists $i$ with $d_i=0.5$, and homotopy equivalent to a sphere of dimension $n(\mb{e},\mb{f})-2$ otherwise.
\end{lemma}

\begin{proposition}
For $n>2$, the component categories $\vec{\pi}_0(\partial\Box^n;\alpha, \mc{F}_{\infty})_*^*,\; \alpha\neq 0,$ all agree with the extension category of the partial order category $\mc{P}_n$.
\end{proposition}

\begin{proof}
The proof relies entirely on properties of (psp) homotopy flows established in Section \ref{s:homflowprop}.
\begin{enumerate}
\item For $k>1$, a (neutral) psp homotopy flow leaves cells of the form $\mb{e}_k=0^k*^{n-k}$, resp.\ $\mb{f}_k=1^k*^{n-k}$ (regardless of the order of the coordinates) invariant; moreover, also their closures and complements. In fact, Lemma \ref{lem:hpytype} shows that leaving one or both of the cells to the future (or the past) changes the homotopy type of $\vec{T}(\partial\Box^n)_{\mb{e}_k}^{\mb{f}_k}$. Apply Lemma \ref{lem:pathcon} and Lemma \ref{lem:noniso}.
\item Now consider $k=1$ and a cell $\mb{e}_1$ of type $0*^{n-1}$: From 1.\ above and Corollary \ref{cor:downarrow}, we conclude that every psp future homotopy flow leaves the set 
$\bigcap \downarrow\! \mb{f}_{n-1}^i$ invariant (with $\mb{f}_{n-1}^i=(1,\dots ,1, *_i, 1,\dots ,1)$). This intersection is equal to the union of all cubes $e$ with no $1$-coordinate at all (here we use $n>2$!). But each cell of type $\mb{e}_1$ is a maximal cube in that set with respect to the partial order. All cubes below $\mb{e}_1$ are of type $\mb{e}_k,\; k>1,$ and are thus invariant; cf 1.\ above. Hence $\mb{e}_1$ is so, as well.

A cell of type $\mb{f}_1=1*^{n-1}$ cannot be connected by a future psp flow with any of the (invariant) cubes in its future. Assume there is a psp future homotopy flow $H$ connecting one of the lower boundary cubes $10^i*^{n-i-1}, \; i>0,$ with $\mb{f}_1$. Then, using Lemma \ref{lem:lr}, $H$ would connect a cell of type $1^20*^{n-3}$ with one of the invariant cells $\mb{f}_k=1^k*^{n-k},\; k>1$. Contradiction!
\item Now consider cubes of type $\mb{e}=0^i*^j1^k,\; i,j,k\ge 1$. Then  $\mb{e}$ is contained in the closure of a cube of type $0^i*^{n-i}$ (invariant by (1) and (2) above). We conclude from Lemma \ref{lem:invariant} that a psp future homotopy flow $H$ departing $\mb{e}$ will end in a cube of type $0^i*^{j-l}1^{k+l},\; l\ge 0$. If $l> 0$, then Lemma \ref{lem:lr} tells us that $H$ would connect a cube of type $*^j1^{i+k}$ (invariant by 1.\ and 2.\ above) with a different cube. Contradiction!\\
 Likewise, a vertex $\mb{e}$ of type $0^{n-k}1^k,\; 1<k<n,$ does not leave $\mb{e}$ under a psp homotopy flow $H$ with $H_0=id$. If so, it would have to end in a cube of type $0^i*^j1^k,\; j>0$. Lemma \ref{lem:lr} implies that the same psp homotopy connects a ``complementary'' cube of type $0^j*^i1^k$ with a cube of type $*^{n-k}1^k$. This contradicts (1) and (2) above. 
\item Since all cubes of the types considered in (3) above are invariant under a psp homotopy flow, none of them can be reached from another cube under a psp homotopy flow $H$ departing from a different cube in its past.
\end{enumerate}
For past pair components, one may apply similar arguments - or consider the reverse d-space that is d-homeomorphic to the original one: Just exchange  $0$ s and  $1$s!
\end{proof}
Remark that the components in this case correspond to products of reachable components of the fundamental category considered in \cite{FGHR:04} (for $n=3)$ 
and for Ziemia\'{n}ski's stable components \cite{Ziemianski:18}.

\subsection{Spaces with directed loops}\label{ss:dirloops}
\subsubsection{Directed circle}\label{sss:dc} The directed circle $\vec{S}^1$ is the pre-cubical set with one $0$-cell and one $1$-cell. Its geometric realization $|\vec{S}^1|$ is a circle on which directed paths proceed counter-clockwise, ie., they are images of non-decreasing paths under the universal covering $\exp : \vec{\mb{R}}\to |\vec{S}^1|$.  All trace spaces $\vec{T}(\vec{S}^1)_x^y,\; x,y\in S^1$, are homotopy equivalent to the discrete space indexed by the non-negative integers $\mathbf{N}_0$.

A directed degree one map $f: S^1\to S^1$ homotopic to the identity induces a homotopy equivalence $\vec{T}(f): \vec{T}(\vec{S}^1)_x^y\to \vec{T}(\vec{S}^1)_{fx}^{fy}$ on \emph{all} trace spaces if and only if it is a directed homeomorphism: Since $f\simeq id$, it is onto. Assume $fx=fy=z$ for $x\neq y$. One of the counter-clockwise arcs from $x$ to $y$, resp.\ from $y$ to $x$ maps to a constant path, the other (without restriction the one from $y$ to $x$), to a d-path from $z$ to itself with winding number one.  But then $\vec{T}(f):\vec{T}(\vec{S}^1)_y^x\to\vec{T}(\vec{S}^1)_z^z$ misses the component given by constant maps!

For every two elements $x,y\in S^1$, there is a future psp homotopy flow $H$ of counter-clockwise rotations (hence directed homeomorphisms) with $H_1(x)=y$. As a result, all pairs $(x,x)$ on the diagonal $\Delta$ are contained in the same component. The diagonal is of course invariant under any map. On the other hand, we have just seen that no pair $(x,y)$ in the complement $\Gamma =(S^1\times S^1)\setminus\Delta$ can be connected to an element on the diagonal by a psp d-map (a homeomorphism!).

Let $(x,y), (x',y')$ denote two pairs in the complement $\Gamma$ of the diagonal. After a rotation, we may assume that $x=x'$. There is a directed homeomorphism (image of a piecewise linear map on $\mb{R}$ under the exponential map) that fixes $x$ and maps $y$ into $y'$ or $y'$ into $y$.
 
As a result, the future pair category of $\vec{\pi}_0(\vec{S}^1;+,\mc{F})_*^*$ has two objects given by the diagonal $\Delta$ and its complement $\Gamma:= (S^1\times S^1)\setminus \Delta$. Endomorphisms on both objects correspond to the non-negative integers $\mb{N}_{\ge 0}$. The morphisms $Mor(\Delta , \Gamma)$ correspond to $\mb{N}_{\ge 0}$, as well, whereas $Mor(\Gamma , \Delta)$ corresponds to the positive integers $\mb{N}$; composition corresponds to addition. Note that $\Delta$ and $\Gamma$ are not isomorphic!

All other pair component categories $\vec{\pi}_0(\vec{S}^1;\alpha,\mc{F})_*^*$ are isomorphic among each other.

\begin{remark}
\begin{enumerate}
\item The result for the directed circle is slightly different from the one obtained in \cite{Raussen:07} where we departed from the extension category of the fundamental category and localized weakly invertible \emph{extensions} instead of d-maps.
\item Remark, that the pair component category \emph{does} distinguish between a directed interval $J\subset\mb{R}$ (Section \ref{sss:int}) and the directed circle $\vec{S}^1$.
\end{enumerate}  
\end{remark}

\subsubsection{Directed torus} Now consider the directed $n$-torus $\vec{T}^n=(\vec{S}^1)^n,\; n\in\mb{N}$. All trace spaces $\vec{T}(\vec{T}^n)$ are homotopy equivalent to a discrete space indexed by $\mb{N}_{\ge 0}^n$. Let us again verify that a psp-d-map $f$ homotopic to the identity is necessarily a homeomorphism. It must be onto since it preserves the fundamental homology class. Assume $f(x)=f(y)=z$ for $x,y\in T^n, x\neq y$. Consider a d-path in $T^n$ from $x$ via $y$ to $x$ of $i$-degree $d_i=0$ if $x_i=y_i$ and $d_i=1$ if $x_i\neq y_i$. The two pieces map under $f$ to paths of multidegree $\delta= (\delta_1,\dots ,\delta_n)$, $\varepsilon = (\varepsilon_1,\dots , \varepsilon_n)$ each; note that $\delta_i + \varepsilon_i = (d_1,\dots ,d_n)$. The maps $\vec{T}(f): \vec{T}(\vec{T}^n)_x^y\to\vec{T}(\vec{T}^n)_z^z$ and $\vec{T}(f):\vec{T}(\vec{T}^n)_y^x\to\vec{T}(\vec{T}^n)_z^z$ have as image subsets whose homotopy types correspond to $\delta +\mb{N}_{\ge 0}^n$, resp.\ $\varepsilon +\mb{N}^n_{\ge 0}$. They cannot both be surjective, ie have image all homotopy types corresponding to $\mb{N}_{\ge 0}^n$.

\begin{proposition}\label{prop:torus}
The pair component category $\vec{\pi}_0(\vec{T}^n;\alpha,\mc{F})_*^*$ is isomorphic to the product $\prod_1^n\vec{\pi}_0(\vec{S}^1;\alpha,\mc{F})_*^*$ This is a category with $2^n$ objects in the set $\{\Delta , \Gamma\}^n$. If one lets $\Delta$ correspond to $0$ and $\Gamma$ to $1$, and hence an object to a bit vector $\mb{d}\in \{ 0,1\}^n$, then $Mor (\mb{d}, \mb{d}')$ is a product of factors $\mb{N}_{\ge 0}$ and $\mb{N}$. A factor corresponds to $\mb{N}$ exactly for pairs $(d_i,d_i')=(1,0)$. Composition corresponds to coordinatewise addition.
\end{proposition}

\begin{proof}
According to Proposition \ref{prop:prod}, we need only show that the quotient functor from the product $\prod_1^n\vec{\pi}_0(\vec{S}^1; \alpha, \mc{F})_*^*\to\vec{\pi}_0(\vec{T}^n;\alpha,\mc{F})_*^*$ is injective on objects, as well. To achieve this, we will show that every inessential d-map $F:\vec{T}^n \to\vec{T}^n$ is a product of d-homeomorphisms $f_i: \vec{S}^1\to\vec{S}^1,$ $1\le i\le n$, ie $F=f_1\times\cdots\times f_n$:

Let $F:\vec{T}^n \to\vec{T}^n$ denote an inessential d-map with $F(x_1,\mb{x}_2)=(y_1,\mb{y}_2)$ and $F(x_1,\mb{x}'_2)=(y'_1,\mb{y}'_2)$, $x_1, y_1, y'_1\in S^1, \mb{x}_2, \mb{x}'_2, \mb{y}_2, \mb{y}'_2\in T^{n-1}$ and $y_1\neq y'_1$. (If we had an inessential map $G: \vec{T}^n \to\vec{T}^n$ relating these pairs in the reverse direction, we may consider $F=G^{-1}$ since $G$ has to be a directed homeomorphism). After a coordinatewise rotation, we may assume that $(y_1,\mb{y}_2)=(x_1,\mb{x}_2)$ -- and $y'_1\neq x_1$. Let $\sigma=(c_{x_1}, \sigma_2)$ denote a d-path from $(x_1,\mb{x}'_2)$ to $(x_1,\mb{x}_2)$, constant in the first variable. Consider the induced diagram of trace spaces
\[\xymatrix{\vec{P}(\vec{T}^n)_{(x_1,\mb{x}_2)}^{(y'_1,\mb{y}'_2)}\ar[r]^{*F(\sigma )} & \vec{P}(\vec{T}^n)_{(x_1,\mb{x}_2)}^{(x_1,\mb{x}_2)}\ar@{^{(}->}[r] &\Omega (T^n; (x_1,\mb{x}_2))\\
\vec{P}(\vec{T}^n)_{(x_1,\mb{x}_2)}^{(x_1,\mb{x}'_2)}\ar[u]^{\vec{T}(F)}\ar[r]^{*\sigma} & \vec{P}(\vec{T}^n)_{(x_1,\mb{x}_2)}^{(x_1,\mb{x}_2)}\ar[u]_{\vec{T}(F)}\ar@{^{(}->}[r] & \Omega (T^n; (x_1,\mb{x}_2))\ar[u]_{\Omega (F)}
}\]  
The image of the lower extension map $*\sigma$ contains loops of multidegree $(0,\mb{d})$.
The image of the upper extension map $*F(\sigma )$ map can only contain loops of multidegree $(d_1,\mb{d})$ with $d_1>0$. This contradicts the fact that an endo map on $T^n$ that is homotopic to the identity has to  preserve multidegrees.
\end{proof}

\begin{remark}

In contrast to what happens to previously considered component categories of the fundamental category (\cite{FGHR:04,GH:07}), this example shows that localized and component categories constructed from pair categories by localizing the effects of psp-d-maps make sense also for d-spaces with \emph{non-trivial directed loops}.
\end{remark}

We postpone the investigation of the pair component categories for the d-space corresponding to Dubut's example to Section \ref{sss:compDubut}.

\section{Pre-cubical sets and order pair components}\label{s:precub}
In general, it seems to be hard to determine the pair component categories introduced in Section \ref{sss:pcc} algorithmically. For pre-cubical sets, cf Section \ref{sss:pcs}, we will now define and describe a related finer component category that arises by inverting \emph{fewer} morphisms and that is easier to comprehend and to determine.  
\subsection{Interval induced maps on a pre-cubical set}\label{ss:intind}
The following construction is essentially contained in \cite{Ziemianski:18}: Let $h:\vec{I}\to\vec{I}$ denote a reparametrization of the unit interval $I$, i.e, a non-decreasing and surjective continuous map. All such reparametrizations form a convex and hence contractible space (in the topology inherited from the compact-open topology). The homeomorphisms within the reparametrizations (the strictly increasing ones) form likewise a convex and hence contractible subspace. Particular reparametrizations are of \emph{future type} if  $h(t)\ge t$ for all $t\in I$, resp.\ of \emph{past type} if $h(t)\le t,\; t\in I$. Algebraically, interval reparametrizations form a monoid under composition (with those of future, resp.\ past type as submonoid) and homeomorphisms form a group (with subgroups of future, resp.\ past type).

Consider a pre-cubical set $X$. Every interval reparametrization $h$ can be used to construct an endo-d-map $h_X$ on the geometric realization $|X|$ defined by $[c, (t_1,\dots ,t_n)]\mapsto [c;(ht_1,\dots ,ht_n)]$. By definition, $h_X$ is a \emph{cube-preserving} map. It is well-defined with respect to the boundary relations on $X$; for this to be true, in general, it is necessary to use the \emph{same} reparametrization $h$ on each coordinate.  

An endo-d-map $h$ in $X$ will be called \emph{interval induced} if it can be described as $h_X$ for a suitable interval reparametrization $h$.

\begin{remark}
\begin{enumerate}
\item For a Euclidean complex, cf Section \ref{ss:Euclid}, one may be less careful and select different reparametrizations for each of the $n$ coordinates.
\item If $h$ is a homeomorphic d-map with inverse $h^{-1}$, then $h_X$ and $(h^{-1})_X=h_X^{-1}$ are inverse homeomophic endo-d-maps.
\item It is clear that the interval induced endo-d-maps on $X$ form a submonoid  of the monoid $d(X)$ of all endo-d-maps; the interval induced endo-d-homeo\-mor\-phisms on $X$ form a subgroup $\Sigma_{\mc{I}}^{\alpha}(X)$ of the group of all endo-d-homeomorphisms.
\end{enumerate}
\end{remark}

\begin{lemma}\label{lem:intindiness}
An interval induced endo-d-map $h_X$ induced by an interval homeomorphism $h$ is neutrally inessential. If it is of future type (resp.\ past type), it is future (resp.\ past) inessential: $\Sigma^{\alpha}_{\mc{I}}(X)\subset\Sigma^{\alpha}_{\mc{F}}(X)$.
\end{lemma}

\begin{proof}
A linear reparametrization homotopy $H:\vec{I}\times I\to \vec{I}$  arises by convex combination of the identity map $id_I$ and the chosen reparametrization $h$. It consists of d-maps $H_t$, and it is an increasing/decreasing homotopy (in the $I$-variable) if $h$ is of future, resp.\ of past type. 
The homotopy $H$ induces the homotopy flow $H_X:X\times I\to X$ given by $H_X([c,e],t)=([c,H(e,t)])$ on $X$ ending at $h_X$.

If $h$ is a homeomorphism (ie injective), it induces a neutral, future, resp.\ a past homotopy flow $H_X: X\times I\to X,\; H_X(-,t)=(H_t)_X(-)$ that consists of d-homeo\-mor\-phisms; in particular, of homotopy equivalences. Each induced map $(H_t)_X: X\to X$ induces thus homeomorphisms on path and trace spaces (with the map induced by the inverse homeomorphism as inverse). Hence it is $\mc{F}$-inessential.
\end{proof}

\emph{Question.} Is the statement in Lemma \ref{lem:intindiness} true for general interval induced endo-d-maps?

Like in Section \ref{sss:pcc}, the monoid $\Sigma^{\alpha}_{\mc{I}}(X)$ of $\alpha$-interval induced d-homeomorphisms on $X$, $\alpha = +,-,0,\pm$, gives rise to wide subcategories (same pair objects) $\Sigma_{\mc{I}}^{\alpha}(X)\subset \Sigma_{\mc{F}}^{\alpha}(X)\subset d(X)$ resp.\ $\Sigma_{\mc{I}}^{\alpha}E\vec{\pi}_1(X)\subset \Sigma_{\mc{F}}^{\alpha}E\vec{\pi}_1(X)\subset dE\vec{\pi}_1(X)$. In the remaining part of the paper, we will focus on the localized category $\Sigma_{\mc{I}}^{\alpha}E\vec{\pi}_1(X)[\Sigma^{\alpha}_{\mc{I}}(X)^{-1}]$ and the (pair) component category, denoted $\vec{\pi}_0(X;\alpha, \mc{I})_*^*$. 

The inclusion of inverted subcategories  $\Sigma^{\alpha}_{\mc{I}}(X)\subset \Sigma^{\alpha}_{\mc{F}}(X)$ gives rise to quotient functors\\ $\Sigma^{\alpha}_{\mc{I}}E\vec{\pi}_1(X)[\Sigma^{\alpha}_{\mc{I}}(X)^{-1}]\to \Sigma_{\mc{F}}^{\alpha}E\vec{\pi}_1(X)[\Sigma^{\alpha}_{\mc{F}}(X)^{-1}]$ between localized categories resp.\ $\vec{\pi}_0(X;\alpha, \mc{I})_*^*\to\vec{\pi}_0(X;\alpha, \mc{F})_*^*$ between quotient categories. In the following, we will show that, for a finite pre-cubical set $X$, the ``finer'' category $\vec{\pi}_0(X;\alpha, \mc{I})_*^*$ has finitely many objects and conclude that this also holds for the pair component categories $\vec{\pi}_0(X;\alpha, \mc{F})_*^*$.

\subsection{Order equivalence and order pair components}
\begin{definition}
\begin{enumerate}
\item Two vectors $\mb{z}=[z_i],\mb{z}'=[z_i']\in\mb{R}^k$ are called \emph{order equivalent} if and only if $z_i R z_j\Leftrightarrow z_i' R z_j'$ for all $1\le i,j\le k$ and $R\in\{<, =, >\}$.
\item Let $\mb{x},\mb{x}'\in\mb{R}^n,\; \mb{y},\mb{y}'\in\mb{R}^m$. The pairs $(\mb{x},\mb{y})$ and $(\mb{x}',\mb{y}')$ are called \emph{order equivalent}
if and only if the vectors $[\mb{x},\mb{y}]$ and $[\mb{x}',\mb{y}']\in\mb{R}^{n+m}$ are order equivalent.
\end{enumerate}
\end{definition}

\begin{lemma}\label{lem:order}
\begin{enumerate}
\item Two vectors $\mb{z}=[z_i],\mb{z}'=[z_i']\in\mb{R}^k$ are order equivalent if and only if there exists a d-homeomorphism $h:\vec{I}\to\vec{I}$ with $h(z_i)=z'_i$.
\item Let $X$ denote a pre-cubical set with an $n$-cell $c$ and an $m$-cell $d$. Let $\mb{x}, \mb{x}'\in I^n, \mb{y}, \mb{y}'\in I^m$ such that $(\mb{x},\mb{y})$ and $(\mb{x}',\mb{y}')$ are order equivalent.
Then there exists an interval induced endo-d-homeomorphism $h_X$ with $h_X([c,\mb{x}])=([c,\mb{x}'])$ and $h_X([d,\mb{y}])=[d,\mb{y}']$.
\end{enumerate}
\end{lemma}

\begin{proof}
A d-homeomorphism $h$ produces clearly order equivalent vectors $\mb{z}=[z_i]$ and $\mb{z}'=[hz_i]$. Given order equivalent vectors $\mb{z}, \mb{z}'$, one may choose a piecewise linear d-homeomorphism with $h(z_i)=z_i'$. (2) follows from (1).
\end{proof}

The following lemma concerning least upper bounds $\mb{x}\vee\mb{x'}$ resp.\ greatest lower bounds $\mb{x}\wedge\mb{x'}$ is straightforward: 

\begin{lemma}\label{lem:order2}
If two pairs $(\mb{x},\mb{y})$ and $(\mb{x}',\mb{y}')$ are order equivalent, then they are both order equivalent to their least upper bounds $(\mb{x}\vee\mb{x}', \mb{y}\vee\mb{y'})$  and greatest lower bounds $(\mb{x}\wedge\mb{x}', \mb{y}\wedge\mb{y}')$, as well. As a consequence
\begin{enumerate} 
\item There exists a $(\Sigma^+_{\mc{I}})$-zig-zag morphism $(\mb{x},\mb{y})\to (\mb{x}\vee\mb{x}',\mb{y}\vee\mb{y}')\leftarrow (\mb{x}',\mb{y}')$.
\item There exists a $(\Sigma^-_{\mc{I}})$-zig-zag-morphism $(\mb{x},\mb{y})\leftarrow (\mb{x}\wedge\mb{x}',\mb{y}\wedge\mb{y}')\to (\mb{x}',\mb{y}')$.
\end{enumerate}
\end{lemma}

\begin{proposition}\label{prp:ordereq}
Let $X$ denote a pre-cubical set with an $n$-cell $c$ and an $m$-cell $d$. Let $\mb{x},\mb{x}'\in I^n$ and $\mb{y},\mb{y}'\in I^m$.  
\begin{enumerate}
\item If the pairs $(\mb{x},\mb{y})$ and $(\mb{x}',\mb{y}')$ are order equivalent, then the pairs $([c,\mb{x}], [d,\mb{y})])$ and  $([c,\mb{x}'], [d,\mb{y}']))$ are situated in the same component in $\vec{\pi}_0(X;\alpha, \mc{I})_*^*$ for all $\alpha$ and hence also in $\vec{\pi}_0(X;\alpha, \mc{F})_*^*$, for all $\alpha$ and $\mc{F}$. In particular, $\vec{T}(X)_{[c,\mb{x}]}^{[d,\mb{y}]}$ and $\vec{T}(X)_{[c,\mb{x}']}^{[d,\mb{y}']}$ are homeomorphic and thus homotopy equivalent.
\item If the pairs $([c,\mb{x}],[d,\mb{y}])$ and $([c,\mb{x}'],[d,\mb{y}'])$ are contained in the same component object in the component category $\vec{\pi}_0(X;\alpha, \mc{I})_*^*$, then $(\mb{x},\mb{y})$ and $(\mb{x}',\mb{y}')$  are order equivalent. 
\item The order pair categories $\vec{\pi}_0(X;\alpha, \mc{I})_*^*,\; \alpha\in\{ +,-,\pm ,0\}$, agree and will be denoted just by $\vec{\pi}_0(X; \mc{I})_*^*$.
\end{enumerate}
\end{proposition}

\begin{proof}
This follows immediately from Lemma \ref{lem:order} and Lemma \ref{lem:order2}.
\end{proof}

\subsection{Order subdivisions and the order pair component category}
Proposition \ref{prp:ordereq} allows to give fairly explicit descriptions of the order pair category $\vec{\pi}_0(X; \mc{I})_*^*$ of a pre-cubical set $X$: A cube $I^k$ subdivides into $k$-simplices given by inequalities $0\le x_{i_1}\le x_{i_2}\le\dots\le x_{i_k}\le 1$; there are $k!$ such simplices. All possible order relations between coordinates arise by replacing $\le$ by either $=$ or by $<$; each replacement (ie each \emph{ordered partition}) gives rise to interiors 
$\mathring{\Delta}$ of subsimplices $\Delta$ of this subdivision. If $k=n+m$, inclusions $[1:n]\hookrightarrow [1:n+m]$, resp.\ $[1:m]\hookrightarrow [1:n+m]$ induce simplicial projections $I^{n+m}\to I^n$ resp.\ $I^{n+m}\to I^m$ (by restriction of a partition to $n$, resp.\ $m$ coordinates). 

To get hold on an \emph{object} in the order pair category $\vec{\pi}_0(X; \mc{I})_*^*$ of $X$, you need to fix a pair of cubes $c\in X_n, d\in X_m$ and the interior of a subsimplex $\Delta_c^d$ of the simplicial subdivision of $|c|\times |d|\cong I^{n+m}$ just described. 
Furthermore, you need to make sure that there exist $(\mb{s},\mb{t})\in\mathring{\Delta}_c^d$ such that $\vec{P}(X)_{[c,\mb{s}]}^{[d,\mb{t}]}\neq\emptyset$. By Proposition \ref{prp:ordereq}, this condition is independent of the choice of base points $\mb{s}, \mb{t}$ in $\mathring{\Delta}_c^d$ ; for example, one may choose the barycenters of each subsimplex $\Delta_c^d$. But beware: These do, in general, not agree with the pair of barycentres of its projections within  $I^n$, resp.\ $I^m$.

\emph{Future (extension) morphisms} from $\Delta_c^d$ to $\Delta_c^{d'}$ are determined by $\vec{\pi}_1(X)_{[d,\mb{t}]}^{[d',\mb{t}']}$ for $(\mb{t},\mb{t}')\in\mathring{\Delta}_d^{d'}$ and $\vec{P}(X)_{[d,\mb{t}]}^{[d',\mb{t}']}\neq\emptyset$; well-determined up to natural homeomorphisms by Lemma \ref{lem:order2}. Composition of future morphisms arises from composition in the fundamental category between matching representatives; you may have to change base point pairs by an element of $\Sigma_{\mc{I}}(X)$ before representatives match! Similarly for \emph{past} morphisms.  

\begin{remark}\label{rem:unique}
The path space between two points in the same cube is contractible or empty; this is in particular true for path spaces between barycenters of a subdivision simplex and a boundary simplex. Hence, there is \emph{at most one} morphism between neighbouring cells given by an extension morphism contained in these two cells. 
\end{remark}

By construction, the functor $\overrightarrow{\pi}_1(X)_*^*: dE\overrightarrow{\pi_1}(X)\to\mathbf{Ho-Top}$ from Section \ref{sss:extcat} extends to $dE\overrightarrow{\pi}_1(X)[\Sigma_{\mc{I}}(X)^{-1}]$; its restriction to $\Sigma_{\mc{I}}E\overrightarrow{\pi}_1(X)[\Sigma_{\mc{I}}(X)^{-1}]$ (analogous to the construction in Section \ref{sss:isc}) factors over the quotient category $\overrightarrow{\pi}_0(X; \mc{I})_*^*$. The quotient functors $\overrightarrow{\pi}_0(X; \mc{I})_*^*\to\overrightarrow{\pi}_0(X;\alpha,\mc{F})_*^*$, $\alpha =+,-,0,\pm$, from Section \ref{ss:intind} are all surjective on objects and full; but rarely faithful.

For a finite-dimensional pre-cubical set $X$, the order pair component can be ``over-appro\-xi\-ma\-ted'' by a discrete full \emph{sub}category of $E\vec{\pi_1}(X)$: For an $n$-dimensional complex $X$, we choose as objects all pairs $[c,\mb{s}], [d,\mb{t}]$ such that all coordinates $s_i$ of $\mb{s}$ and $t_j$ of $\mb{t}$ are fractions $\frac{l}{2n+1},\; 0\le l\le 2n+1$. This choice ensures that every subsimplex $\Delta_c^d$ contains at least one such pair. The projection functor from the arising subcategory of $E\vec{\pi}_1(X)$ into $\vec{\pi}_0(X; \mc{I})_*^*$ is onto on objects and fully faithful, hence an equivalence of categories. Using Remark \ref{rem:unique} for (2) below, we conclude:

\begin{corollary}
Let $X$ denote a finite pre-cubical set.
\begin{enumerate}
\item $\vec{\pi}_0(X; \mc{I})_*^*$ and thus all pair component categories $\vec{\pi}_0(X; \alpha, \mc{F})_*^*$ have finitely many objects.
\item If $X$ does not admit any non-trivial directed loops, then all the above mentioned categories have finite sets of morphisms.
\end{enumerate}
\end{corollary}

\subsection{Euclidean cubical complexes}\label{ss:Euclid}
With the lattice of integer vectors $\mb{Z}^n\subset\mb{R}^n$ as vertices, one forms an infinite pre-cubical set whose geometric realization is $\mb{R}^n$. Every subcomplex $X$ of that complex is called a \emph{Euclidean cubical complex}. By definition, such a Euclidean complex admits only trivial directed loops.

\begin{definition}\label{def:H}
\begin{enumerate}
\item Two elements $(\mb{x},\mb{x}')\in X$ will be called \emph{cube-equi\-va\-lent} if there exists a cube $c\subset X$ such that $\mb{x}, \mb{x}'$ are both contained in the \emph{interior} of $c$. 
\item A d-map $h=(h_1,\dots ,h_n): \mb{R}^n\to\mb{R}^n$ with $h_i(x_1,\dots ,x_i,\dots ,x_n)=x_i$ for every $x_i\in\mb{Z}, 1\le i\le n$, is called \emph{cube-preserving}.
\item If $h(\mb{x})\ge \mb{x}$ (resp.\ $h(\mb{x})\le \mb{x})$, $h$ is said of future type, resp.\ of past type.
\item A cube-preserving d-map $h$ restricts to a cellular endo-d-map $h: X\to X$ from every subcomplex $X$ into itself. These cube-preserving maps form, when restricted to $X$, a (contractible!) monoid. Those of future (resp.\ past) type form contractible submonoids $\mc{H}_n^{\alpha}(X)$.
\end{enumerate}
\end{definition}

\begin{proposition}\label{prp:H-psp}
Let $X$ denote a Euclidean cubical complex. Any cube preserving d-map $h:X\to X$ of future type (resp.\ of past type) is $+\mc{F}_{\infty}$ (resp.\ $-\mc{F}_{\infty}$) inessential.
\end{proposition}

\begin{proof}
Cellwise convex combination between $id_X$ and $h$ defines a homotopy flow $H$ such that every d-map $H_t: X\to X$ is cube-preserving, as well. It is a future (resp.\ past) homotopy flow if $h$ is of future (resp.\ past) type. It is therefore enough to show that an $\alpha$ cube-preserving map $h: X\to X$ is $\mc{F}$-psp; we do that for $\alpha =+$: 

According to Remark \ref{rem:dext}.5 (cf also \cite[(5.3)]{Raussen:07}), the diagram
\[\xymatrix{\vec{T}(X)_{\mb{x}}^{\mb{y}}\ar[r]^{\vec{T}(h)}\ar[dr]_{*H(\mb{y})} & \vec{T}(X)_{h(\mb{x})}^{h(\mb{y})}\ar[d]^{H(\mb{x})*}\\
 & \vec{T}(X)_{\mb{x}}^{h(\mb{y})}}\]
 is homotopy commutative for $(\mb{x},\mb{y})\in\vec{X^2}$. Hence it is enough to show that for every $(\mb{x}, \mb{y})\in\vec{X^2}$ and every d-path
 \begin{enumerate}
 \item $\sigma$ from $\mb{y}$ to $\mb{y}'$ between cube-related elements, the map $*\sigma : \vec{P}(X)_{\mb{x}}^{\mb{y}}\to\vec{P}(X)_{\mb{x}}^{\mb{y}'}$
 \item $\tau$ from $\mb{x}'$ to $\mb{x}$ between cube-related elements, the map $\tau *: \vec{P}(X)_{\mb{x}}^{\mb{y}}\to\vec{P}(X)_{\mb{x}'}^{\mb{y}}$
 \end{enumerate} 
are homotopy equivalences. (The construction makes uses of paths instead of traces, but the naturalization map between traces and d-paths in Euclidean complexes from \cite{Raussen:09}
is a homotopy equivalence.) We prove 1.\ above; the proof of 2.\ is similar.

A map $M^{\mb{y}}: \vec{P}(X)_{\mb{x}}^{\mb{y}'}\to\vec{P}(X)_{\mb{x}}^{\mb{y}}$ is constructed by assigning to $\alpha\in\vec{P}(X)_{\mb{x}}^{\mb{y}'}$ the d-path $\alpha^{\mb{y}}\in\vec{P}(X)_{\mb{x}}^{\mb{y}},\;  \alpha^{\mb{y}}(t)=\alpha (t)\wedge\mb{y}$, ie $\alpha_i^{\mb{y}}(t)=\min \{\alpha_i(t), y_i\}$. Since $\mb{y}$ and $\mb{y}'$ are cube-related, $\alpha (t)$ and $\alpha^{\mb{y}}(t)$ are, for every $t\in I$, contained in the same subcube of $\mb{R}^n$, and hence $\alpha^{\mb{y}}$ is a d-path in $X$.

The composition $\xymatrix{\vec{P}(X)_{\mb{x}}^{\mb{y}}\ar[r]^{*\sigma} & \vec{P}(X)_{\mb{x}}^{\mb{y}'}\ar[r]^{M^{\mb{y}}} & \vec{P}(X)_{\mb{x}}^{\mb{y}}}$  assigns to $\alpha\in\vec{P}(X)_{\mb{x}}^{\mb{y}}$ the path $\alpha * c_{\mb{y}}$, which is homotopic to the identity map. 

For $s\in I$, let $\sigma^s$ denote the d-path $\sigma^s(t):=\sigma (s+(1-s)t)$ from $\sigma (s)$ to $\sigma (1)=\mb{y}'$. We consider the homotopy
$\xymatrix{\vec{P}(X)_{\mb{x}}^{\mb{y}'}\times I\ar[r]^{M^{\sigma (s)}} & \vec{P}(X)_{\mb{x}}^{\sigma (s)}\ar[r]^{*\sigma^s} & \vec{P}(X)_{\mb{x}}^{\mb{y}'}}$ (with parameter $s$); it deforms the map $*\sigma\circ M^{\mb{y}}$ -- for $s=0$ -- into $*c_{\mb{y}'}$ -- for $s=1$. The latter map sends $\alpha\in\vec{P}(X)_{\mb{x}}^{\mb{y}'}$ into $\alpha*c_{\mb{y}'}$, and it is homotopic to the identity map on $\vec{P}(X)_{\mb{x}}^{\mb{y}'}$.
\end{proof}

Proceeding as in Section \ref{ss:intind}, the monoids from Definition \ref{def:H}
give rise to wide subcategories  $\Sigma^{\alpha}_{\mc{H}_n}(X)\subset d(X)\subset dE\vec{\pi}_1(X), \; \alpha =+,-,0,\pm$. 
We can form the localized categories $\Sigma_{\mc{H}_n}^{\alpha}E\vec{\pi}_1(X)[\Sigma_{\mc{H}_n}^{\alpha}(X)^{-1}]$, their quotient categories $\vec{\pi}_0(X;\alpha ,\mc{H}_n)_*^*$ and, using Pro\-po\-sition \ref{prp:H-psp}, quotient functors into $\vec{\pi}_0(X;\alpha ,\mc{F})_*^*$, at least for $\alpha\neq 0$.

Proposition \ref{prp:ordereq} has the following (far simpler) analogue for Euclidean cubical complexes:

\begin{proposition}\label{prp:cubeq}
Let $X$ denote a Euclidean complex and let  $(\mb{x},\mb{y}),\; (\mb{x}',\mb{y}')\in\vec{X^2}$.
\begin{enumerate}
\item If $(\mb{x},\mb{x}')$ and $(\mb{y},\mb{y}')$ are cube equivalent, then $(\mb{x}, \mb{y})$ and  $(\mb{x}', \mb{y}')$ are situated in the same component in $\vec{\pi}_0(X;\alpha ,\mc{H}_n)_*^*$. In particular, $\vec{T}(X)_{\mb{x}}^{\mb{y}}$ and $\vec{T}(X)_{\mb{x}'}^{\mb{y}'}$ are homotopy equivalent.
\item If the pairs $(\mb{x},\mb{y})$ and $(\mb{x}',\mb{y}')$ are situated in the same component object in the component category $\vec{\pi}_0(X;\alpha ,\mc{H}_n)_*^*$, then $(\mb{x},\mb{x}')$ and $(\mb{y},\mb{y}')$  are cube equivalent. 
\item The pair categories $\vec{\pi}_0(X;\alpha, \mc{H}_n)_*^*,\; \alpha\in\{ +,-,\pm ,0\}$, agree and will be denoted just by $\vec{\pi}_0(X,\mc{H}_n)_*^*$. There are quotient functors $\vec{\pi}_0(X,\mc{H}_n)_*^*\to\vec{\pi}_0(X; \alpha,\mc{F}_{\infty})$ for all $\alpha$.
\end{enumerate}
\end{proposition}

For the proof of Proposition \ref{prp:cubeq}, we need the following lemma. We will write $(\mb{x},\mb{y})\sim_{\mc{H}}(\mb{x}',\mb{y}')$ if they are contained in the same component; here for $\alpha =+$. We shall write $\mb{x}<<\mb{x}'$ if $x_i<x_i'$ for all $i$.  

\begin{lemma}\label{lem:cr}
\begin{enumerate}
\item Let $\mb{x}, \mb{x}''$ be cube-related and $\mb{x}''<<\mb{x}$. Let $(\mb{x}, \mb{y})\in\vec{X^2}$.\\ Then $(\mb{x}'',\mb{y})\sim_{\mc{H}}(\mb{x},\mb{y})$.
\item Let $\mb{y}, \mb{y}''$ be cube-related, $(\mb{x}'', \mb{y})\in\vec{X^2}, \mb{y}''<<\mb{y}$. Then $(\mb{x}'',\mb{y})\sim_{\mc{H}}(\mb{x}'',\mb{y}'')$.
\end{enumerate}
\end{lemma}

\begin{proof} of Lemma \ref{lem:cr}:
To show (1), we construct a cube-preserving d-map $h=h_1\times\dots\times h_n:\mb{R}^n\to\mb{R}^n$ (of future type) with $h(\mb{x}'')=\mb{x}$ and $h(\mb{y})=\mb{y}$ as product of piecewise linear d-maps $h_i, 1\le i\le n$, given by
\[ h_i(t)=\begin{cases} t & t\le\ceil{x_i} \mbox { or } t\ge\max\{ y_i, \ceil{x_i}+1\} \\ x_i' & t=x_i'' 
\end{cases};\]
by assumption $x_i''<y_i$. To show (2), a similarly constructed cube-preserving d-map fixes $\mb{x}''$ and sends $\mb{y}$ into $\mb{y}''$. 
\end{proof}

\begin{proof} of Proposition \ref{prp:cubeq}:
\begin{enumerate}
\item Choose $\mb{x}''<<\mb{x},\mb{x}'$ in the interior of the same cell as $\mb{x}$ and $\mb{y}, \mb{y}'<<\mb{y}''$ in the same cell as $\mb{y}$. According to Lemma \ref{lem:cr}, we obtain the following chain of equivalences:\\ $(\mb{x},\mb{y})\sim_{\mc{H}}(\mb{x}'',\mb{y})\sim_{\mc{H}}(\mb{x}'',\mb{y}'')\sim_{\mc{H}}(\mb{x}'',\mb{y}')\sim_{\mc{H}}(\mb{x}',\mb{y}')$.\\
A similar construction works for $\alpha = -$.
\item is obvious, since d-maps in $\mc{H}_n$ preserve (interiors of) cubes.
\item follows from (1) and (2).
\end{enumerate}
\end{proof}

\begin{corollary}
Let $X\subset\mb{R}^n$ denote a Euclidean cubical complex.
\begin{enumerate}
\item Components in $\vec{\pi}_0(X,\mc{H}_n)_*^*$ can be indexed by reachable pairs of cubes $(c, d)$. 
\item The category $\vec{\pi}_0(X,\mc{H}_n)_*^*$ is isomorphic to the full subcategory of $E\vec{\pi}_1(X)$ the objects of which are the \emph{pairs of barycenters} of reachable cubes.
\item Components in each of the categories $\vec{\pi}_0(X;\alpha ,\mc{F})_*^*$ are unions of such pairs. 
\item The categories $\vec{\pi}_0(X,\mc{H}_n)_*^*$ and hence $\vec{\pi}_0(X;\alpha ,\mc{F})_*^*$ have finitely many objects and morphisms.
\end{enumerate}
\end{corollary}

\begin{remark}
Contrary to what happens for general cubical complexes, for Euclidean cubical complexes we can, as in the proof of Proposition \ref{prp:cubeq}, construct inessential endo-d-maps \emph{fixing just one of the end points} and  leading to inessential extension morphisms; compare Lemma \ref{lem:homflow}.4. This is the reason why, for these complexes, it may be unecessary to distinguish between effects of inessential d-maps and of inessential \emph{extensions}, as they were used in previous work on component categories \cite{FGHR:04,GH:07}.
\end{remark}

\subsection{Dubut's example revisited}\label{sss:compDubut}
Finally, we analyze component categories in the case of the cubical complex $D$ from Section \ref{ss:Dubut}, which is not a Euclidean complex. Nevertheless, some of the tools from the preceeding sections come in helpful. It turns out that in this case, the $\alpha\mc{F}$-inessential d-maps $f: D\to D$ have a quite specific form that we deduce for $\alpha =+$: For that purpose, we have to consider a more precise decomposition of the space $D$. Using natural homeomorphisms identifying the four 2-cells with $I^2$, we define:

\begin{enumerate}
\item $a_1=\{ [A;(1,t)]|\; t<1\} =\{ [B_1; (0,t)]|\; t<1\}$;
\item $a_2=\{ [A;(t,1)]|\; t<1\} =\{ [B_2; (t,0)]|\; t<1\}$;
$a^+=[A; (1,1)]$;
\item $\llcorner\! A=A\setminus (a_1\cup a_2\cup a^+)$;
\item $b_1=\{ [B_1;(t,1)]|\; t>0\}$; 
$b_2=\{ [B_2;(1,t)]|\; t>0\}$;
\item $\lrcorner\! B_1=B_1\setminus (a_1\cup b_1\cup a^+)$;
$\ulcorner\! B_2= B_2\setminus (a_2\cup b_2\cup a^+)$;
\item $c_1=\{ [C;(1,t)]|\; t<1\}$; 
$c_2=\{ [C;(t,1)]|\; t<1\}$; 
$c^+=[C; (1,1)]$;    
\item $\llcorner\! C= C\setminus (c_1\cup c_2\cup c^+)$.         
\end{enumerate}

\begin{center}\begin{figure}[h]\label{fig:Dub3}
\begin{tikzpicture}
  \draw (0,0) -- (4,0) -- (4,2) -- (2,2) -- (2,4) -- (0,4) -- (0,0);
\draw (2.5,2.5) -- (4.5,2.5) -- (4.5,4.5) -- (2.5,4.5) -- (2.5,2.5);
\draw (0,2) -- (2,2);
\draw (2,0) -- (2,2);
\node at (1,1) {$A$};
\node at (1,3) {$B_2$};
\node at (3,1) {$B_1$};
\node at (3.5,3.5) {C};
\node at (1,2) {$a_2$};
\node at (2,1) {$a_1$};
\node at (2.2,2.1) {$a^+$};
\node at (3,2) {$b_1$};
\node at (2,3) {$b_2$};
\node at (3.5,4.5) {$c_2$};
\node at (4.5,3.5) {$c_1$};
\node at (4.7,4.6) {$c^+$};
\draw[line width=0.6mm, color=red] (4,0) -- (4,2);
\draw[line width=0.6mm, color=blue] (0,4) -- (2,4);
\draw[line width=0.6mm, color=red] (2.5,2.5) -- (2.5,4.5);
\draw[line width=0.6mm, color=blue] (2.5,2.5) -- (4.5,2.5);
\end{tikzpicture}
\caption{Decomposition of the cubical complex $D$}
\end{figure}
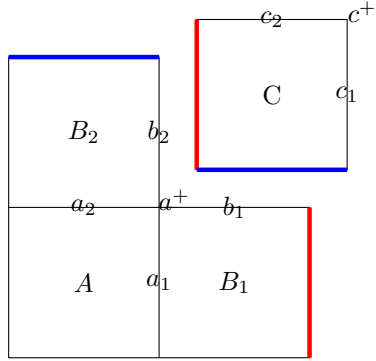\end{center}

\begin{proposition}\label{prp:Dubut}
A $+\mc{F}$-inessential d-map $f: D\to D$ has the following properties: 
\begin{enumerate}
\item $f$ preserves $A, a_i, \; i=1,2, a^+, D\setminus A, C, \lrcorner\! B_1\cup b_1\cup C, \ulcorner\! B_2\cup b_2\cup C$.
\item The restrictions of $f$ to $A$ and to $C$ agree, ie\\ $f([A; (x_1,x_2)])=f([C; (x_1,x_2)]),\; x_1,x_2\in I$.
\item On $A$, resp.\ on $C$, $f=f|_A=f|_C$ is a product map $f=f_1\times f_2$ for suitable d-maps $f_i: \vec{I}\to\vec{I}$ that satisfy $t\le f_i (t), t\in I$.
\item These d-maps $f_i$ are d-homeomorphisms $f_i: [0,1]\to [d_i,1]$ for some $0\le d_i<1$. 
\item There are d-maps $g_i: I^2\to [0, 1+d_i],\; i=1,2$, with $g_1(0,t)=0=g_2(t,0), g_1(1,t)=d_1, g_2(t,1)=d_2$ and $(s,t)\le g_i(s,t),\; i=1,2,$ such that $f([B_1;(x_1,x_2)])=$\\ $[B_1\cup C;(g_1(x_1,x_2),f_2(x_2))]$ and $f([B_2;(x_1,x_2)])=[B_2\cup C;(f_1(x_1),g_2(x_1,x_2))]$ with $f_1, f_2$ as in \emph{(3)}. 
\item $f$ preserves $\llcorner\! A, c^+, b_1\cup c_2, b_2\cup c_1, \lrcorner\! B_1\cup\llcorner\! C, \ulcorner\! B_2\cup\llcorner\! C$.
\end{enumerate}
\end{proposition}

\begin{proof}
\begin{enumerate}
\item  The vertex $a^+$ is a future branch point. By Corollary \ref{cor:downarrow}.1, both that point, its past $\downarrow\! a^+=A$ and the complement of $A$ are preserved. Moreover, $f$ preserves $\uparrow\! a_i\cap \downarrow\! a^+=a_i$ and $C=\uparrow\! C$. Also the subsets $\lrcorner\! B_1\cup b_1\cup C$ and $\ulcorner\! B_2\cup b_2\cup C$ are equal to their own future.
\item For a chosen point $[C;(y_1,y_2)]$, we consider the set of all points $[A;(x_1,x_2)]$ in its past\\ $\downarrow\! [C;(y_1,y_2)]$ and decompose it according to whether $(x_1,x_2)$ belongs to
\begin{description}
\item[$Q_0(y_1,y_2)$]=$[0,y_1]\times [0,y_2]$. For $(x_1,x_2)\in Q_0(y_1,y_2)$, trace space $\vec{T}(D)_{[A;(x_1,x_2)]}^{[C;(y_1,y_2)]}$ has two path components. 
\item[$Q_1(y_1,y_2)$]=$]y_1,1]\times [0,y_2]$. For $(x_1,x_2)\in Q_1(y_1,y_2)$, trace space $\vec{T}(D)_{[A;(x_1,x_2)]}^{[C;(y_1,y_2)]}$ is path-connected; all paths in it intersect $B_1\setminus\!\{ a^+\}$, but not $B_2\setminus\!\{ a^+\}$.
\item[$Q_2(y_1,y_2)$]=$[0,y_1]\times ]y_2,1]$. For $(x_1,x_2)\in Q_2(y_1,y_2)$, trace space $\vec{T}(D)_{[A;(x_1,x_2)]}^{[C;(y_1,y_2)]}$ is path-connected; all paths in it intersect $B_2\setminus\!\{ a^+\}$, but not $B_1\setminus\!\{ a^+\}$. 
\end{description}
Let $f([C,(x_1,x_2)])=[C;(y_1,y_2)]$. Since $f$ is a psp d-map, its restriction $f_A$ to $A$ must map $Q_j(x_1,x_2)$ into $Q_j(y_1,y_2),\, j=1,2,3.$  By continuity, it maps the one point set $\{ (x_1,x_2)\} =\bar{Q}_1(x_1,x_2)\cap\bar{Q}_2(x_1,x_2)$ into $\bar{Q}_1(y_1,y_2)\cap\bar{Q}_2(y_1,y_2)=\{ (y_1,y_2)\}.$
\item The horizontal d-path from $[A; (0,x_2)]$ to $[C; (0,x_2)]$ through $B_1$ maps under $f$ to the horizontal d-path from $[A; f(0,x_2)]$ to $[C; f(0,x_2)]$ through $B_1$; in particular, $f_2(t,x_2)=f_2(0,x_2)$ for $t\in I$. Likewise for $f_1$ using vertical d-paths through $B_2$. The existence of a future d-homotopy from $id_D$ to $f$ requires $t\le f_i(t),\; t\in I$.
\item By (1) above, $f_i(1)=1, \; i=1,2$. Let $d_i:=f_i(0)\ge 0$. Assume there exists $0\le x_1<x_1'\le 1$ with $f_1(x_1)=f_1(x_1')$. Then, for every $x_2\in I$, the space $\vec{T}(D)_{[A;(x_1',x_2)]}^{[C;(x_1,x_2)]}$ has one component whereas $\vec{T}(D)_{[A;(f_1(x_1'),f_2(x_2))]}^{[C;(f_1(x_1),f_2(x_2))]}$ has two. Hence $f$ cannot be psp.\\ Similarly for $f_2$.
\item The second component of $f|_{B_1}$ does not depend on $x_1$ and is equal to $f_2$ since\\ $f_2(x_2)=f_2([A; (x_1,x_2)])\le f_2([B_1;(x_1,x_2)])\le f_2([C];(x_1,x_2)])=f_2(x_2),\; x_1\in I$.
\item Since $f_i$ is a homeomorphism, it preserves the upper boundary $1$ and its complement. 
\end{enumerate}
\end{proof}

Like in Section \ref{ss:intind}, but now with a coherent prescribed order on the first and second coordinate, we may compare coordinates of start and end points by the relations $=, <, >$. For example, a \emph{decoration} $(=<)$ indicates that the first coordinates of start and end point agree whereas the second coordinate of the start point is less than the second coordinate of the end point.

\begin{proposition}
The pair component category $\vec{\pi}_0(D; +,\mc{F})$ has -- independently of $\mc{F}$ -- objects of the form: Reachable pairs in
\begin{enumerate}
\item $(\llcorner\! A, \llcorner\! A)$ with decorations $(==), (=<), (<=)$ and $(<<)$;
\item $(\llcorner\! A, a_i),\; i=1,2$; with decorations $=$ and $<$ in one coordinate;
\item $(\llcorner\! A, a^+)$;
\item $(a_i, a_i),\; i=1,2,$ with decorations $<$ and $=$;
\item $(a_i, a^+),\; i=1,2$;
\item $(a^+,a^+)$;
\item $(\llcorner\! A, \llcorner\! C)$ with decorations $(<<), (<=), (=<), (==)$;
\item $(\llcorner\! A, \lrcorner\! B_1\cup\llcorner\! C)$ with decorations $(><), (>=)$; 
\item $(\llcorner\! A, \ulcorner\! B_2\cup \llcorner\! C)$ with decorations $(<>), (=>)$;
\item $(\llcorner\! A, (b_1\cup c_2))$ and $(\llcorner\! A, (b_2\cup c_1))$;
\item $(\llcorner\! A, c^+)$;
\item $(a_i, \lrcorner\! B_i\cup\llcorner\! C), \; i=1,2$, with decorations $=$ and $<$;
\item $(a_i, b_1\cup c_2), (a_i, b_2\cup c_1), (a_i,c^+),\; i=1,2$;
\item $(a^+,  b_1\cup c_2), (a^+, b_2\cup c_1), (a^+,c^+)$;
\item $(\lrcorner\! B_1\cup\ulcorner\! B_2 \cup \llcorner\! C, \lrcorner\! B_1\cup\ulcorner\! B_2 \cup \llcorner\! C)$ -- no decorations;
\item $(\lrcorner\! B_1\cup \llcorner\! C, b_1\cup c_2),\; \lrcorner\! B_2\cup \llcorner\! C, b_2\cup c_1)$;
\item $(\lrcorner\! B_1\cup \llcorner\! C, c^+),\; \lrcorner\! B_2\cup \llcorner\! C, c^+)$;
\item $(b_1\cup c_2, b_1\cup c_2), (b_2\cup c_1, b_2\cup c_1)$;
\item $(b_1\cup c_2, c^+), (b_2\cup c_1, c^+)$;
\item $(c^+,c^+)$
\end{enumerate} 
and the inherited extension morphisms.
\end{proposition}

\begin{proof}
Most of the distinctions follow from the fact that $f|A=f|C=f_1\times f_2$ is a product of \emph{homeomorphisms}, hence points with non-equal coordinates cannot be identified with points with equal coordinates (unlike what may happen for a homotopy equivalence). In order to construct a psp-d-map $f:D\to D$ establishing equivalence of pairs of points in one of the equivalence classes, one may choose $0\le d_i<1$ and d-homeomorphisms $f_i:[0,1]\to [d_i,1],\; i=1,2,$ with $t\le f_i(t),\; t\in I$ and use them to define the map $f$ on $A\cup C$ according to Proposition \ref{prp:Dubut}(2) -(4). To extend that map to the remaining cells, choose d-maps $g_i: I^2\to [0,1+d_i]$ as in Proposition \ref{prp:Dubut}(5) and define $f$ on $B_1\cup B_2$ accordingly. A linear future d-homotopy connects $id_D$ with the resulting endo map $f$ on $D$. In general, a zig-zag of such d-homotopies is needed.

It may be a bit surprising that decorations do not turn up in case (15). One may extend a map $f_i: [0,1]\to [d_i,1]$ from Proposition \ref{prp:Dubut}(4) to d-maps $g_i:[-1,1]\to [-1,1]$ (for the first coordinate on $B_1$, resp.\ the second on $B_2$); this map needs only be injective on $[0,1]$. Using such maps, we establish that two pairs $[C;(c_1,c_2), (c_1,c_2)]$ and $[C; (c_1,c_2), (c_1',c_2')]$ are $\mc{F}$-equivalent as follows:
\begin{eqnarray*}
[C;(c_1,c_2), (c_1,c_2)] & \sim [B_1;(b_1,c_2), (b'_1,c_2)] & \sim [C;(c_1,c_2), (c'_1,c_2)]\\ & \sim [B_2;(c_1,b_2), (c'_1,b'_2)] & \sim[C;(c_1,c_2), (c'_1,c'_2)].
\end{eqnarray*}
This implies that all pairs in $\vec{\llcorner\! C^2}$ are equivalent to each other, and hence all pairs in case (15). 
\end{proof}

Remark that inessential $+\mc{F}$ maps do not preserve all cells; hence $\vec{\pi}_0(D; +,\mc{F})$ has fewer component objects than $\vec{\pi}_0(D,\mc{H}_2)_*^*$, cf Section \ref{ss:Euclid}. 

For $\alpha =-$, a similar case-by-case examination exhibits components of type $(C,C)$ with decorations; this time, start points in $A\cup B_i$ can be fused. For $\alpha=\pm$, the subsets $B_i$ cannot be fused with neither $A$ nor $C$.

\section{Conclusion and future work}
\subsection{Summary} Inessential homotopy flows and inessential d-maps yield a coherent framework for 
comparing path spaces with variable end points within a given directed space. Localizing their contribution to categories with pairs of points as objects transforms them into \emph{iso}morphisms; this is justified since the trace space functor lets them correspond to isomorphisms in the homotopy category. The resulting quotient categories were shown to have finitely many (though often a huge number of) objects when the underlying directed space is a finite cubical complex, even if the space admits directed loops. In many examples, in particular for the example from Section \ref{ss:Dubut} motivating this paper in the first place, the quotient categories retain essential information about dependence of the path/trace spaces on their end points in a compressed way. 
  
\subsection{Future work: Relations to other constructions}\label{ss:relations}
As the motivating example (Section \ref{ss:Dubut}) shows, previously studied component categories \cite{FGHR:04,GH:07,FGHMR:16} do not always deliver categories with a countable number of objects, even when the directed space has the nice structure of a finite cubical complex. This paper advises a way to overcome this default - and making the constructions of so-called ``Natural homology'' \cite{DGG:15} (and its precursor in \cite{Raussen:07}) applicable in this more general setting. In contrast to other approaches, at least in principle, the definitions are suitable also in cases where the space admits directed loops. This option has only been investigated in detail only in a few concrete cases in Section \ref{ss:dirloops}; further development is desirable.

Ziemia\'{n}ski found a different way to overcome the shortcomings of previous work on components by defining and investigating \emph{stable components} \cite{Ziemianski:18} that partition the directed space itself (instead of the space of reachable pairs). Path spaces between pairs of points within two components are, in general, not invariant up to $\mc{F}$-equivalence, but they become so after a stabilization process based on a number of well-motivated axioms. Path spaces between points in a given pair of components may vary but they stabilize when allowing large enough targets (resp.\ small enough sources). The resulting stable components are easier to determine than ours.  
It would be interesting to find out whether the components in this paper can be produced by partitioning reachable pairs of Ziemia\'{n}ski's stable components into smaller pieces.

In a different direction, it is a challenge to develop a directed version of Michael Farber's topological complexity \cite{Farber:08}. The definitions can easily be modified by requiring directed paths, but one needs a better understanding of the end point map $e: \vec{P}(X)\to \vec{X^2}$ and its partial sections. In the directed case, this map is, in general far from being a fibration, and the (Schwarz genus) methods from the classical theory are not available. First steps have been taken in \cite{GFS:18}; relations to methods from this paper might be helpful in future developments.

\subsection{Future work: Functoriality. Towards directed homotopy equivalences?}\label{ss:dhe}

A d-map $F: X\to Y$ between two d-spaces does not give rise to any relation between the spaces of endo-d-maps $\vec{C}(X,X)$ and $\vec{C}(Y,Y)$; cf Section \ref{sss:endo} for the notation. Instead, one needs a \emph{pair} $F: X\to Y$ and $G: Y\to X$ of d-maps giving rise to the maps $(G,F)_{\#}: \vec{C}(X,X)\to \vec{C}(X,X),\; f\mapsto F\circ f\circ G$ and $(F,G)_{\#}: \vec{C}(Y,Y)\to \vec{C}(Y,Y),\; g\mapsto G\circ g\circ F$. We need further properties allowing us to relate the ``homotopy dynamics'' on the two d-spaces:

\begin{definition}\label{df:dhe}
A pair of d-maps $F: X\to Y, G:Y\to X$ is called an $\alpha\mc{F}$-\emph{equivalence} if there exist $\alpha\mc{F}$ psp homotopy flows $H_1:X\times I\to X$ connecting $id_X$ with $G\circ F$ and $H_2:Y\times I\to Y$ connecting $id_Y$ with $F\circ G$.
\end{definition} 

In particular, for $\mc{F}=\mc{F}_{\infty}$, the map $F$ is an ordinary weak homotopy equivalence. 

\begin{lemma}\label{lem:indeq}
An $\alpha\mc{F}$-\emph{equivalence} $(F,G)$ induces $\mc{F}$-equivalences\\ $\vec{T}(G\circ F):\vec{T}(X)_{x_1}^{x_2}$ $\to\vec{T}(X)_{GFx_1}^{GFx_2}, \vec{T}(F\circ G):\vec{T}(Y)_{y_1}^{y_2}\to\vec{T}(Y)_{FGy_1}^{FGy_2},$\\ $\vec{T}(F): \vec{T}(X)_{x_1}^{x_2}\to\vec{T}(Y)_{Fx_1}^{Fx_2}$ and $\vec{T}(G): \vec{T}(Y)_{y_1}^{y_2}\to\vec{T}(X)_{Gy_1}^{Gy_2},\; x_i\in X, y_i\in Y$.
\end{lemma}
\begin{proof}
The first two properties follow from the definitions. For the last claims, apply the 2-out-of-6 property for (our) $\mc{F}$-equivalences.
\end{proof}

It is tempting to call a d-map $F: X\to Y$ satisfying the requirements in Definition \ref{df:dhe} a directed $\alpha\mc{F}$ equivalence; and for $\mc{F}=\mc{F}_{\infty}$ - the (weak) homotopy equivalences - a directed (weak) $\alpha$ homotopy equivalence. But this definition is still not quite satisfactory. It is not clear that this notion
\begin{itemize}
\item satisfies a 2-out-of-3-properties
\item leaves directed topological complexity \cite{GFS:18} invariant, and
\item that it behaves well with respect to components.
\end{itemize}
The follow-up paper \cite{Raussen:19} responds to these challenges with an adjustement of Definition \ref{df:dhe} as point of departure.

\end{document}